\documentclass[12pt]{article}
\usepackage{amsmath, amsfonts, amsthm, amssymb, verbatim}
\usepackage{graphicx}
\usepackage[mathcal]{euscript}
\usepackage{amstext}
\usepackage{amssymb,mathrsfs}
\usepackage[latin1]{inputenc}
\newcommand{\R}{\mathbb R}
 \newcommand{\Z}{\mathbb Z}
 \newcommand{\N}{\mathbb N}

\newcommand{\ve}{\varepsilon}
\newcommand{\vp}{\varphi}
\newcommand{\ess}{\,{\rm ess}}
\newcommand \loc    {\text{loc}}
\DeclareMathOperator  \sgn {sgn} 

\newcommand{\dive}{{\rm div}}

\newcommand{\clg}[1]{{\mathcal{#1}}}

\newcommand{\ol}[1]{{\overline{#1}}}

\newtheorem{theorem}{Theorem}[section]

 \newtheorem{remark}[theorem]{Remark}
\newtheorem{lemma}[theorem]{Lemma}

\newtheorem{definition}[theorem]{Definition}

\begin{document}
%
\title{Initial-boundary value problem for 
\\
stochastic transport equations}

\author{Wladimir Neves$^1$, Christian Olivera$^2$}

\date{}

\maketitle

\footnotetext[1]{ Instituto de Matem\'atica, Universidade Federal
do Rio de Janeiro, C.P. 68530, Cidade Universit\'aria 21945-970,
Rio de Janeiro, Brazil. E-mail: {\sl wladimir@im.ufrj.br.}

}


%
\begin{abstract}
This paper concerns the Dirichlet initial-boundary value problem for
stochastic transport equations with non-regular coefficients. 
First, the existence and uniqueness of the strong stochastic traces is proved. 
The existence of weak solutions relies on 
the strong stochastic traces, and also
on the passage from the 
Stratonovich into It\^o's formulation
for bounded domains.
Moreover, the uniqueness is established 
without the divergence of
the drift vector field bounded from below.

 \end{abstract}
%
\maketitle

%

\section {Introduction} \label{Intro}

A great deal of attention has recently been given to the study of stochastic 
partial differential equations.
We are interested in random description of physical problems, 
where the probabilistic term appears as a perturbation of the velocity vector field. 
In this direction, it was S. Ogawa \cite{SO}
who initiated the analysis of wave propagation in random media. 

In this article we establish global existence and uniqueness of
solutions for the stochastic linear transport equations (SLTE for short) in bounded domains. 
Namely, we consider 
the following initial-boundary value problem: Given a standard Brownian motion
$B_{t} = (B_{t}^{1},...,B _{t}^{d} )$ in $\mathbb{R}^{d}$,  
find $u(t,x) \in \R$, satisfying 
\begin{equation}\label{trasport}
 \left \{
\begin{aligned}
    &\partial_t u(t, x,\omega) + \Big(b(t, x) + \sigma \frac{d B_{t}}{dt}(\omega) \Big) \cdot \nabla u(t, x,\omega)= 0,
    \\[5pt]
    &u{|_{t=0}}= u_0, \qquad 
    u|_{\Gamma_T}= u_b, 
\end{aligned}
\right .
\end{equation}
with $(t,x) \in U_T:= [0,T] \times U$, where
 $T> 0$ is any fixed real number,
$U$ is an open and bounded domain of $\R^d$ $(d \in \N)$, 
$\omega \in \Omega$ is an element of the probability space $(\Omega, \mathbb{P}, \clg{F})$,
and the stochastic integration is taken in the Stratonovich sense. The parameter
$\sigma= 1$ most of the time, and equals zero when we talk about \eqref{trasport} in the 
deterministic case.
Moreover, we denote 
by $\Gamma$ the $C^2$-boundary of $U$, with 
the outside normal field to $U$ at $r \in \Gamma $ denoted by $%
\mathbf{n}(r)$, and define $\Gamma_T:= (0,T) \times \Gamma$.

Here, we assume that the initial and boundary data respectively $u_0$, $u_b$ 
are measurable and bounded functions with respect to the usual measures, that is, Lebesgue 
(denoted by $dx$, or $d\xi$, etc.) and Hausdorff (denoted by $\clg{H}^{d-1}(r)$ or $dr$) tensor $dt$. 

The vector field $b: (0,T) \times \R^d \to \R^d$, called drift, satisfies the following conditions: 
For any $q> 2$ and some non-negative functions $\alpha, \, \gamma \in L^1_\loc(\R)$,
\begin{equation}
\label{con1}
    b \in  L^q((0,T); BV_{\loc}(\R^d; \R^d)), \quad \dive b\in L^{1}_{\loc}((0,T) \times \R^d),  
\end{equation}
\begin{equation}
\label{BCONDUNQ}
  |b(t,x)| \leq \alpha(t), \qquad \dive b(t,x) \leq \gamma(t).
\end{equation}
Let us remark that, we assume $q> 2$ 
in order to use the machinery developed for the Ladyzhenskaya-Prodi-Serrin condition, see Remark \ref{stochastic influx}.
Also the $BV$ regularity guarantee the restriction of $b(t,\cdot)$ on $\Gamma$ in the sense of trace, and
allow us to make use of commutators (in 
particular to show the existence of
strong stochastic trace).
Moreover, we follow Funaki \cite{Fu} where the main tool to show existence of weak solutions for
regular-coefficients is a time reversed process
(see Section \ref{preliminares}),
thus the vector function $b(t,\cdot)$ is defined in the all space $\R^d$. In any case, if $b(t,\cdot)$ is just defined
in $U$, then we may use an extension theorem for $BV$ functions. 

Now, let us briefly recall that
the problem \eqref{trasport} has been treated  for the case  $U=\mathbb{R}^{d}$  by many authors, 
both for the deterministic and stochastic cases,
see for instance \cite{Falnanas},  \cite{ambrosio}, \cite{ambrisio2} \cite{ambrisio2}  \cite{DL},  \cite{FNO}, \cite{FGP2}, \cite{Ku},  \cite{Moli}, \cite{WNCO2}.
%
DiPerna, Lions in \cite{DL} (deterministic case)
proved that $W^{1,1}$ 
spatial regularity of $b(t,x)$
(together with a condition of boundedness on the divergence) is enough to ensure uniqueness of weak solutions. 
Moreover, they deduced the existence, uniqueness and stability results for 
ordinary differential equations with rough coefficients from corresponding results on the associated linear transport equation. 
Ambrosio in \cite{ambrosio} following the same strategy in \cite{DL}, but applying a measure-theoretic framework,
generalized the results to the case where the coefficients have only bounded variation regularity by considering the continuity equation. 
%
Then, Flandoli, Gubinelli and Priola
in \cite{FGP2} proved that, 
the stochastic problem is better behaved than the deterministic one (the first result in this direction).
They obtained wellposedness of the stochastic problem for an H\"older continuous drift term, 
with some integrability conditions on the divergence.

The premiere researches of linear transport equations (deterministic case) in bounded domains
was done by Bardos  \cite{BARDOS}. In that extended paper, Bardos considered the regular case 
($b$ has Lipschitz regularity), and established the correct 
understanding of how the Dirichlet boundary condition should be assumed, 
where the notion of the influx boundary zone is important, that is
\begin{equation}
\Gamma _{T}^{-}:=\big\{ (t,r) \in \Gamma _{T}:( b \cdot 
\mathbf{n})(t,r)< 0 \big \} . \label{flux}
\end{equation}
Then, we mention the work of Mischler \cite{Misch}, who considered weak solutions for the Vlasov equation (instead of 
the transport equation) posed in bounded domains. 
In that paper, the trace problem for linear transport type equations
is discussed in details. 
One observes that, if $u$ is not sufficiently regular, in particular we look
for measurable and bounded solutions, the restriction to negligible Lebesgue sets is not, a priori, defined. Therefore,
one has to deal with the traces theory to ensure the correct notion of the Dirichlet boundary condition. In the same 
direction as Mischler \cite{Misch}, Boyer \cite{Bo} established the trace theorems 
with respect to the measure $\mu$ defined on $\Gamma_T$ as 
\begin{equation}
\label{MUMEASURE}
  d\mu :=(b \cdot \mathbf{n})\,d{r}dt,
\end{equation}  
and showed the existence and uniqueness of solutions for the transport equation
using the Sobolev framework of DiPerna, Lions \cite{DL}. 
More recently, 
Crippa, Donadello, Spinolo \cite{CDS2} studied the initial-boundary value problems for
continuity equations with total bounded variation coefficients. We stress that, there does not 
exist strong trace results for (deterministic) transport equations with non-regular coefficients.
Indeed, let us recall the following counter-example given by 
Neves, Panov and Silva in \cite{WNEPJS}, (with slight modifications). 
First, we define the 
$2D$ vector field $b(t,x)$ as follows:

1. If $t \in [1, 3/2]$, then
$$
     b(t,x)= B(t-1,x) \quad \text{in the square $x= (x_1,x_2) \in [0,1]^2$,}
$$
where the vector $B= B(t,x)$ (for $t \in [0, 1/2]$ only) is taken from Lemma 10 of the
paper by Kneuss, Neves (see \cite{OKWN}),
and it is extended by periodicity for all $(x_1,x_2) \in \R^2$.
Again, we reproduce it here for the convenience of the reader.
Define $c\in L^{\infty}([-1/2,1/2]^2;\mathbb{R}^2)$ and $d\in L^{\infty}([-1/2,1/2]\times [-1/4,1/4];\mathbb{R}^2)$ respectively by
$$c(x):=\left\{
\begin{array}{cl}
(0,8 x_1)& \text{if $|x_2|<|x_1|<1/2$,}\\
(-8 x_2,0)& \text{if $|x_1|<|x_2|<1/2$,}\\
0&\text{elsewhere,}
\end{array}
\right.
$$
and
$$d(x):=\left\{
\begin{array}{cl}
(0,4 x_1)& \text{if $|2x_2|<|x_1|<1/2$,}\\
(-16 x_2,0)& \text{if $|x_1|<|2x_2|<1/2$,}\\
0&\text{elsewhere.}
\end{array}
\right.
$$
Then $\operatorname{div}c=0$ in $\clg{D}'((-1/2,1/2)^2)$ and the normal component of $c$ is $0$ across $\partial
(-1/2,1/2)^2$. Similarly, $\operatorname{div} d=0$ in $\clg{D}'((-1/2,1/2)\times (-1/4,1/4))$ and the normal component of $d$ is $0$ across $\partial
\left[(-1/2,1/2)\times (-1/4,1/4)\right]$. Moreover, the vectors $c$ and $d$
are $BV(\R^2; \R^2)$, (see \cite{OKWN}, p.150). 
Notice, that the flow $\xi^c(t,\cdot)$ generated by the field $c$ is a ``square" rotation in the ``angle'' $2\pi t$
(for $t= k/4, k \in \mathbb{Z})$. Similarly, the flow $\xi^d(t,\cdot)=R^{-1}\xi^c(t,\cdot)R$ is the ``rectangle'' rotation, here $R(x_1,x_2)=(x_1,2x_2)$.
Therefore, we define $B= B(t,x) \in L^{\infty}$ by
$$B(t,x):=\left\{
\begin{array}{cl}
d(x_1-1/2,x_2-1/4) &\text{for $0\leq t<1/4$ and $x\in [0,1]\times [0,1/2], $}\\
d(x_1-1/2,x_2-3/4) &\text{for $0\leq t<1/4$ and $x\in [0,1]\times [1/2,1], $}\\
-c(x_1-1/2,x_2-1/2) &\text{for $1/4\leq t \leq1/2$ and $x\in [0,1]\times [0,1],$ }
\end{array}
\right.
$$
and then extend it for all $x\in\R^2$ by the periodicity. Since, the normal component of $B(t,\cdot)$ is zero across $\partial [0,1]^2$, then  $\operatorname{div} B(t,\cdot)= 0$ in $\clg{D}'(\mathbb{R}^2)$. The flow $(t,\xi^B(t,x))$ generated by the field $(1,B)$ satisfies the following relation in the square $(0,1)^2$
\begin{equation}\label{pan1}
\xi^B(1/2,x)=\left\{\begin{array}{lcr} (x_1/2,2x_2), & , & 0<x_2<1/2, \\ ((x_1+1)/2,2x_2-1) & , & 1/2<x_2<1. \end{array} \right.
\end{equation}
It follows from (\ref{pan1}) that for a $1$-periodic function $\phi(y)$, $u(\xi^B(1/2,x))=\phi(2x_2)$ if $u(x)=\phi(x_2)$. 

2. For $t \in [(2/3)^k, (2/3)^{k-1})$, $(k= 1,2, \ldots)$,
we define $b(t,x)$ by the scaling
$$
   b(t,x)= (3/4)^k \ B((3/2)^k t - 1, 2^k x).
$$
So that for $0\le t\le\tau_k\doteq\frac{1}{2} (2/3)^k$ the corresponding flow on the square $2^{-k}(0,1)^2$ has the form $((2/3)^k+t,2^{-k}\xi^B((3/2)^k t,2^kx))$. In fact, denoting $\xi^b(t,x)=2^{-k}\xi^B((3/2)^k t,2^kx)$, we have
$$
\begin{aligned}
\frac{d}{dt}\xi^b(t,x)&= (3/4)^k b((3/2)^k t,\xi^B((3/2)^k t,2^kx))
\\
   &= (3/4)^k b((3/2)^k t,2^k\xi^b(t,x))= b((2/3)^k+t,\xi^b(t,x)),
\end{aligned}
$$
and $\xi^b(0,x)=x$. 
In view of the periodicity of $B(t,\cdot)$ the flow on the shifted squares $2^{-k}(z+(0,1)^2)$, $z\in\Z^2$, has the form
$\xi^b(t,x)=2^{-k}z+\xi^b(t,x-2^{-k}z)$. Observe that 
\begin{equation}\label{pan2}
\xi^b(\tau_k,x)= 2^{-k}\xi^B(1/2,2^kx).
\end{equation}
This means that the time $\tau_k$ is just enough to complete all the rotations. We underline that the time scale $(2/3)^k$ decays slower than the space scale $2^{-k}$, which allow to use the smaller field (which is taken with the small factor $(3/4)^k$ just equaled to the ratio of the space and the time scales).

Relation (\ref{pan2}) implies that for a $2^{-k}$-periodic $\phi(y)$ 
\begin{equation}\label{pan3}
u(\xi^b(\tau_k,x))= \phi(2x_2), \quad \mbox{ where } u(x)=\phi(x_2).
\end{equation}
Notice that the function $\phi(2y)$ is $2^{-k-1}$-periodic. 

3. For $t> 3/2$ we set $b \equiv 0$.

Choosing the nonconstant 1-periodic function $\phi(y)$,
we may construct (by the usual method of characteristic) a solution of the transport equation
$$
   \partial_t u(t,x) + b(t,x) \cdot \nabla u(t,x)= 0,
$$
which satisfies $u(3/2, x)= \phi(x_2)$. Applying (\ref{pan3}) for $k=0,1,\cdots$, we arrive at the relations
$u( (2/3)^k, x)= \phi(2^{k+1} x_2)$, hence $u$ has no
strong trace at the plane $\{t= 0\}$. On the other hand, by the construction
$\beta(u)$ is a solution of the same transport equation for any $\beta \in C(\R)$, hence the renormalization
property is satisfied, which implies uniqueness. 

Also related to strong traces for deterministic transport equations, we would like 
to thank an anonymous referee for bringing to our attention the following 
result in \cite{FEFFERMAN}. Adopting the setting of weak measure-value solutions, there 
exists a continuous divergence-free drift $b \in C_c([0,\infty) \times \R^2)$, and a 
time-dependent measure $\eta(t)$, such that the Hausdorff dimension of 
${\rm supp}\  \eta(t)> 1$, for each $t> 0$, but  the Hausdorff dimension of 
${\rm supp}\  \eta(0)= 0$. Therefore, the strong trace at the plane $\{t= 0\}$ is not 
achieved in this class of solutions, since we expect that, by the transport 
property of the equation, the support of the solutions ($t> 0$ fixed) 
should be maintained at the plane $\{t= 0\}$. 

\medskip
Let us now focus on the stochastic case. First, Funaki
in \cite{Fu} studied  the random transport equation 
in bounded domains with regular coefficients.
To the knowledge of the authors, nothing has already been done 
for stochastic transport equations in bounded 
domains for low regularity coefficients. 
Actually, different from the deterministic setting, we could not 
use the idea of the influx zone, where the boundary data 
is prescribed. The solutions to \eqref{trasport} 
will be constructed via the idea of stopped backward process,
see \eqref{SRP}, which was used in \cite{Fu} and well explored
by Constantin, Iyer \cite{PCGI}, related to Navier-Stokes equations
in domains with boundaries, where the velocity vector
field has Lipschitz regularity.

\medskip
In this article, we deal with the problem \eqref{trasport} and show the 
existence and uniqueness of  weak  $L^{\infty}$-solutions for Dirichlet data. 
The initial-boundary value problem is much harder to solve than the Cauchy one, 
for instance, the solvability in the weak sense for the Cauchy problem is easily established 
under the mild assumption of local integrability for $b$ and $\dive b$, see \cite{WNCO2}.
On the other hand, the existence result established here on
bounded domains relies strongly on the 
strong stochastic trace result obtained in Section \ref{TRACE}, that is to say, the trace of
a distributional solution $u$ of \eqref{trasport} 
is a function $\gamma u \in L^\infty([0,T] \times \Gamma \times \Omega)$
(see Definition \ref{deftrace}).
This result of strong traces does no follow (necessarily) 
in the deterministic case, we recall from \cite{Bo}, ($p= \infty$), that 
$\gamma u \in L^\infty([0,T] \times \Gamma; |\mu|)$, i.e. in weak sense 
($\gamma u$ in duality with $b \cdot \mathbf{n}$).

We have also used to prove the existence of weak solutions, the passage from the 
Stratonovich formulation \eqref{IBVPEXT} into It\^o's one \eqref{IBVPEXTIto}, which is a 
completely new result. 

The uniqueness of weak solutions obtained in this paper does not assume
that the divergence of $b$ is bounded (we have just assumed a boundedness 
from above). Moreover, we only consider a boundedness of $b$ with 
respect to the spatial variable, see \eqref{BCONDUNQ}.

\section{Existence of Weak Solutions} 
\label{CLASSICALSOL}

The main issue in this section is to establish the solvability of system \eqref{trasport}.
We shall assume that $b$ satisfies \eqref{con1}, \eqref{BCONDUNQ}, otherwise 
mentioned explicitly. 

\subsection{Preliminares and Background}
\label{preliminares}

$\circ$ {\bf Weak solutions for regular coefficients} 
\medskip

To begin, let us consider the random differential equation in $\R^d$, that is to say,
given $s \in [0,T ]$  and  $x \in \mathbb{R}^{d}$, we consider 
\begin{equation}\label{itoass}
X_{s,t}(x)= x + \int_{s}^{t}   b(t', X_{s,t'}(x)) \ dt'  +  B_{t}-B_{s},
\end{equation}
where $X_{s,t}(x)= X(s,t,x)$ (also $X_{t}(x)= X(0,t,x)$).
In particular, for $m \in \N$ and $0 < \alpha < 1$, we assume
\begin{equation}
\label{DRIFTREGULAR}
   b \in L^1((0,T); (C^{m,\alpha}(\R^d; \R^d)). 
\end{equation}
It is well known that, under the above regularity 
of the drift vector field $b$, the 
stochastic flow $X_{s,t}$ is a $C^m$ diffeomorphism
(see for example  \cite{Ku2,Ku}). Moreover, the 
inverse $Y_{s,t}:=X_{s,t}^{-1}$ 
satisfies the following backward stochastic
differential equations,
\begin{equation}\label{itoassBac}
Y_{s,t}= y - \int_{s}^{t}   b(t', Y_{t',t}) \ dt'  - (B_{t}-B_{s}),
\end{equation}
for $0\leq s\leq t$. Usually, $Y$ is called the time
reversed process of $X$. 
Then, given $(t,x) \in U_T$ and the time
reversed process $Y_{s,t}$, we consider the set
$
    S= \{ s \in [0,t] / \, Y(s,t,x) \notin U_T \}
$
and define 
\begin{equation}
\label{TAU}
  \kappa(t,x,\omega):= \sup S. 
\end{equation}  
Clearly $S$ could be an empty set, and in this case we set $\kappa= 0$.

To follow, we define $\bar{Y}_{s,t}$ on $\bar{U}$ as
\begin{equation}
\label{SRP}
    \bar{Y}_{s,t}(x):= Y_{s,t}(x) \quad \text{for $s \in [ \kappa, t]$},
\end{equation}
which is called a stopped backward process.
Moreover, we define for each $(t,x) \in U_T$,
the stochastic influx boundary zone, which is to say
$$
  \Gamma^{\rm{in}}(\omega):=
  \big\{ \bar{Y}_{\kappa,t}(x) ; \; \kappa(t,x)> 0 \big\},
$$
and for convenience $\Gamma^{\rm o}:= \Gamma \setminus \Gamma^{\rm {in}}$. 
Finally, we set 
\begin{equation}
   \begin{aligned}
    \mathbf{n^i}&= (- \chi_{\Gamma^{\rm{in}}}) \; \mathbf{n},
    \quad \text{and} \quad
    \mathbf{n^o}&= (1-\chi_{\Gamma^{\rm{in}}}) \; \mathbf{n}.
   \end{aligned}
\end{equation}

From the above considerations, 
we may apply a straightforward computation (see conjointly  Funaki \cite{Fu}, Theorem 3.1) to prove the following  

\begin{lemma}\label{lemaexis} 
For $m\geq 3$, $0< \alpha < 1$, let $u_{0} \in C^{m,\alpha}(\ol{U})$, $u_b \in C^{m,\alpha}(\ol{\Gamma_T})$ 
be respectively initial, boundary data satisfying compatibility conditions,
and  assume \eqref{DRIFTREGULAR}.
Then, the IBVP problem \eqref{trasport}
has a weak (regular-coefficients) $L^\infty$-solution $u(t,.)$ for $0\leq t\leq T$, given 
by
\begin{equation}\label{repre}
u(t,x):=
\left \{
\begin{aligned}
       {u}_{0}(\bar{Y}_{\kappa,t}(x)), & \quad \text{if  $\kappa(t,x)= 0$},
\\[5pt]
       {u}_{b}(\kappa, \bar{Y}_{\kappa,t}(x)), & \quad \text{if  $\kappa(t,x)> 0$},
\end{aligned}
\right.
\end{equation}
which satisfies: 
For each test function $\varphi \in C_c^{\infty}(\R^d)$, the real value process $\int_{U}  
u(t,x)\varphi(x) dx$ has a continuous modification which is a
$\mathcal{F}_{t}$-semimartingale, and for all $t \in [0,T]$, we have $\mathbb{P}$-almost
sure
\begin{equation} 
\label{IBVPEXTREG}
     \begin{aligned}
      \int_U  & u(t,x) \varphi(x) dx= \int_{U} u_{0}(x)\varphi(x) \ dx
      +\int_{0}^{t} \!\! \int_{U}u(s,x) \, b^j(s,x)  \, \partial_j \varphi(x) \ dx ds
\\      
        &+ \int_{0}^{t}\!\!  \int_{U} u(s,x) \, \dive \, b(s,x)\,  \varphi(x) \ dx ds 
           -  \int_{0}^{t} \!\! \int_{\Gamma} \gamma u(s,r) \, \varphi(r) \, b^j \mathbf{n}_{\!j} \, dr ds
\\[5pt]
     &-\int_{0}^{t} \!\! \int_{\Gamma}   \gamma u(s,r) \, \varphi(r)  \, \mathbf{n}_{\!j} \ dr \circ dB_{s}^{j}
   + \int_{0}^{t} \!\! \int_{U} u(s,x)  \; \partial_j \varphi(x) \ dx {\circ}dB_{s}^{j},
   \end{aligned}
\end{equation}
where $ \gamma u \, \mathbf{n}= u_\mathbf{o} \, \mathbf{n^o} - u_b \, \mathbf{n^i}$.
\end{lemma}

\begin{remark}
\label{REMTRACEREGCOEFF}
One remarks that, the trace of the solutions 
$u(t,x)$ defined by \eqref{repre}, say $\gamma u$,
makes sense (see Lemma 4.3 in \cite{Fu}).
Moreover, the boundary data $u_b$ is assumed in \eqref{IBVPEXTREG} 
just on the stochastic influx boundary zone.
\end{remark}

\bigskip
$\circ$ {\bf Distributional  solution} 
\medskip

We begin considering in which sense 
a function $u\in L^{\infty}(U_T \times \Omega)$
is a distributional solution to problem \eqref{trasport},
more precisely we have the following 

\begin{definition}\label{defisoldistri} Let $u_0 \in L^\infty(U)$ 
be given. A stochastic process
$u\in L^{\infty}(U_T \times \Omega)$ is called 
a distributional  $L^{\infty}-$solution of the IBVP \eqref{trasport},
when for each test function $\varphi \in C_c^{\infty}(U)$, the real value process $\int_{U}  
u(t,x)\varphi(x) dx$ has a continuous modification which is a
$\mathcal{F}_{t}$-semimartingale, and for all $t \in [0,T]$, we have $\mathbb{P}$-almost
sure
\begin{equation} \label{DISTINTSTR}
\begin{aligned}
    \int_U u(t,x) \varphi(x) dx &= \int_{U} u_{0}(x) \varphi(x) \ dx
   +\int_{0}^{t} \!\! \int_{U} u(s,x) \ b^i(s,x) \partial_{x_i} \varphi(x) \ dx ds
\\[5pt]
   &\; + \int_{0}^{t} \!\! \int_{U} u(s,x) \, \dive \,b(s,x) \, \varphi(x) \ dx ds 
\\[5pt]
    &\;  + \int_{0}^{t} \!\! \int_{U} u(s,x) \ \partial_{x_i} \varphi(x) \ dx \, {\circ}{dB^i_s}.
\end{aligned}
\end{equation}
\end{definition}

\begin{remark}\label{lemmaito1}  Since distributional solutions and Cauchy problem
can be treated equivalently, 
following Flandoli, Gubinelli, Priola \cite{FGP2}, see Lemma 13, 
we can reformulate equation \eqref{DISTINTSTR} in
It\^o's form as follows: A  stochastic process $u\in L^{\infty}(U_T \times \Omega)$
is  a distributional  $L^{\infty}$ solution  of the SPDE
\eqref{trasport} if, and only if, for every test function
  $\varphi \in C_c^{\infty}(U)$, the process $\int u(t,
  x)\varphi(x)   dx$ has a continuous modification, which is a
$\mathcal{F}_{t}$-semimartingale, and satisfies the following It\^o's formulation 
for all $t \in [0,T]$
\begin{equation} \label{DISTINTITO}
\begin{aligned}
    \int_U u(t,x) &\varphi(x) dx= \int_{U} u_{0}(x) \varphi(x) \ dx
   +\int_{0}^{t} \!\! \int_{U} u(s,x) \ b^i(s,x) \partial_{x_i} \varphi(x) \ dx ds
\\[5pt]
   &+ \int_{0}^{t} \!\! \int_{U} u(s,x)\,  \dive \, b(s,x) \, \varphi(x) \ dx ds 
\\[5pt]
    &+ \int_{0}^{t} \!\! \int_{U} u(s,x) \ \partial_{x_i} \varphi(x) \ dx \stackrel{}{dB^i_s}
    + \frac{1}{2} \int_{0}^{t} \!\!\int_{U} u(s,x) \Delta \varphi(x) \ dx ds.
\end{aligned}
\end{equation}
\end{remark}

\begin{lemma}\label{lemmaexis1}  Under condition \eqref{con1}, \eqref{BCONDUNQ},
there exits a distributional $L^{\infty}$ solution $u$ of the stochastic IBVP
\eqref{trasport}.
\end{lemma}

The proof of the above lemma follows the same arguments,
with minor modifications, as the one for the Cauchy problem,
see Lemma 2.1 in Neves, Olivera \cite{WNCO2}.

\subsection{Strong Stochastic Trace} 
\label{TRACE}

Now we prove the existence and uniqueness of the strong stochastic trace by the existence of distributional 
$L^{\infty}-$solution of the IBVP \eqref{trasport}. 

 \begin{definition}\label{deftrace}  Let $u$ be a  distributional  $L^{\infty}$-solution of the IBVP problem \eqref{trasport}.  
 A stochastic process
$\gamma u \in L^{\infty}([0,T] \times \Gamma \times \Omega)$ is called 
the stochastic trace of the distributional  solution $u$, if for each test function $\varphi \in C_c^{\infty}(\R^d)$, $\int_{\Gamma}  
\gamma u(t,r) \varphi(r) dr$ is an adapted real value process,
which satisfies for any $\beta \in C^{2}(\mathbb{R})$
and all $t \in [0,T]$
\begin{equation}
\label{STOCTRACE}
    \begin{aligned}
    \int_U \beta( u(t,x)) \; \varphi(x) \ dx&= \int_{U} \beta(u_{0}(x)) \; \varphi(x) \ dx
\\[5pt]
   &+\int_{0}^{t} \!\! \int_{U} \beta(u(s,x)) \, b(s,x) \cdot  \nabla \varphi(x) \ dx ds  
   \\[5pt]
  & + \int_{0}^{t}\!\! \int_{U} \beta (u(s,x)) \, \dive \, b(s,x)\,  \varphi(x) \ dx ds
\\[5pt]
  &-  \int_{0}^{t} \!\! \int_{\Gamma} \beta (\gamma u) \, \varphi(r) \, b(s,r)  \cdot \mathbf{n}(r)  \ dr ds
\\[5pt]
   &+ \int_{0}^{t} \int_{U} \beta (u(s,x)) \,  \partial_{x_{i}} \varphi(x) \ dx \,  {\circ}dB_{s}^{i}  
\\[5pt]
  &-  \int_{0}^{t} \!\! \int_{\Gamma} \beta (\gamma u) \, \varphi(r) \,  \mathbf{n}_i(r) \ dr \, {\circ}dB_{s}^{i}.
   \end{aligned}
\end{equation}
\end{definition}

\begin{theorem}\label{PROPtrace}  Assume condition \eqref{con1}, and 
 let $u$ be a  distributional  $L^{\infty}$-solution of the IBVP problem \eqref{trasport}.  
 Then, there exits a unique  stochastic trace $ \gamma u$. 
\end{theorem}

\begin{proof} 1.  Let $u$ be a distributional solution of the transport equation
\eqref{trasport}, and for each $\ve> 0$ set 
$u_{\varepsilon}(t,\cdot)$ the global approximation of $u$ related to the standard mollifier 
$\rho_\ve$ (see Appendix). Let $\psi \in C^\infty_c(U)$ be a positive function,
and consider for any fixed $y \in U$ 
$$
   \varphi(x)= \psi(y) \, \rho_\ve(y + \lambda \, \ve \nabla h(y) - x),
$$   
hence $\varphi$ vanishes on the boundary $\Gamma$. 
Then, we take conveniently $\varphi$ as a test function in \eqref{DISTINTSTR} to obtain 
$$
\begin{aligned}
    u_{\varepsilon}(t,y) &= (u_0\ast_{\mathbf{n}} \rho_\ve)(y)
\\[5pt]
   &+ \int_{0}^{t}\!\! \int_U u(s,z) \,  b(s,z) \cdot \nabla  \rho_{\varepsilon}(y^\ve - z) \ dz ds 
\\[5pt]
  &+  \int_{0}^{t} \!\! \int_U u(s,z) \, \dive \, b(s,z) \, \rho_{\varepsilon}(y^\ve-z) \ dz ds  
\\[5pt]
  &+  \int_{0}^{t} \!\! \int_U u(s,z) \, \partial_{i} \rho_{\varepsilon}(y^\ve-z) \, dz \, {\circ}dB_{s}^{i}.
\end{aligned}
$$
Let $\beta \in C^2(\R)$, and applying  It\^o-Ventzel-Kunita
Formula (see Appendix),  we obtain from the above equation  
$$
  \begin{aligned}
  \beta( u_{\varepsilon}(t,x)) &= \beta(u_0\ast_{\mathbf{n}}  \rho_\ve)(x) 
\\[5pt]
   & +\int_{0}^{t}  \beta^{\prime}(u_\varepsilon(s,x)) \int_U u(s,z) \,  b(s,z) \cdot \nabla \, \rho_{\varepsilon}(x^\ve-z) \ dz ds 
\\[5pt]
  & + \int_{0}^{t} \beta^{\prime}(u_\varepsilon(s,x)) \int_U u(s,z) \, \dive b(s,z) \, \rho_{\varepsilon}(x^\ve-z)\ dz  ds 
\\[5pt]
   &+  \int_{0}^{t}  \beta^{\prime}(u_\varepsilon(s,x))  
   \int_U u(s,z) \,  \partial_{i} \rho_{\varepsilon}(x^\ve-z) dz \,  { \circ}dB_{s}^{i}.
   \end{aligned}
$$   
Following the renormalization procedure, nowadays well known, 
we obtain from an algebraic manipulation 
\begin{equation}
\label{au}
  \begin{aligned}
  \beta( u_{\varepsilon}(t,x)) &- \beta(u_0\ast_{\mathbf{n}} \rho_\ve)(x) 
\\[5pt]
    &+\int_{0}^{t}  b(s,x) \cdot \nabla \beta(u_\ve(s,x)) \ ds 
   +  \int_{0}^{t}  \partial_{i}\beta(u_\varepsilon(s,x))  
     \,  { \circ}dB_{s}^{i}
     \\[5pt]
     &= \int_{0}^{t} \beta^\prime(u_\varepsilon(s,x)) \, \mathcal{R}_{\varepsilon}(b,u) ds +
         \int_{0}^{t} \partial_i\beta(u_\varepsilon(s,x)   \mathcal{P}_{\varepsilon}(u)  \circ dB_{s}^{i},
   \end{aligned}
\end{equation}
where $\mathcal{R}_{\varepsilon}(b,u)$, $ \mathcal{P}_{\varepsilon}(u)$ are commutators type, defined respectively by
$$
   \begin{aligned}
    \mathcal{R}_{\varepsilon}(b,u)&:=(b\nabla ) (\rho_{\varepsilon}\ast_{\mathbf{n}} u )- \rho_{\varepsilon}\ast_{\mathbf{n}}((b\nabla)u),
\\[5pt]
    \mathcal{P}_{\varepsilon}(u)&:=\nabla  (\rho_{\varepsilon}\ast_{\mathbf{n}} u )- \rho_{\varepsilon}\ast_{\mathbf{n}}(\nabla u).
   \end{aligned}
$$ 
	
\medskip
2. Now, we show that $\{\beta(u^\ve)\}$ is a Cauchy sequence in $L^2([0,T] \times \Gamma \times \Omega)$.
For any $\varepsilon_1, \varepsilon_2> 0$, setting $w_{\varepsilon1,2}=  \beta(u_{\varepsilon_1}) 
-  \beta(u_{\varepsilon_2})$, we get from equation \eqref{au}
$$
  \begin{aligned}    
  w_{\varepsilon1,2}(t,x) &- w_{\varepsilon1,2}(0,x) 
\\[5pt]  
  & + \int_{0}^{t}  \ b(s,x) \cdot \nabla w_{\varepsilon1,2}(s,x) \ ds 
 + \int_{0}^{t} \partial_i w_{\varepsilon1,2}(s,x) \circ dB_{s}^{i}
 \\[5pt]
 &= \int_{0}^{t} \mathcal{R}_{\varepsilon1,2}(b,u) ds +  
 \int_{0}^{t}   \mathcal{P}_{\varepsilon1,2}(u)  \circ dB_{s}^{i}, 
   \end{aligned}
$$
where 
$$
      \mathcal{R}_{\varepsilon1,2}(b,u)=  \beta^\prime(u_{\varepsilon_1}) \, \mathcal{R}_{\varepsilon_1}(b,u)  
      -\beta^\prime(u_{\varepsilon_2}) \, \mathcal{R}_{\varepsilon_2}(b,u),
$$  
and
$$
      \mathcal{P}_{\varepsilon1,2}(u)=  \beta^\prime(u_{\varepsilon_1}) \, \mathcal{P}_{\varepsilon_1}(u)  
      -\beta^\prime(u_{\varepsilon_2}) \, \mathcal{P}_{\varepsilon_2}(bu).
$$  
Similarly to item 1, we apply in the above equation the It\^o-Ventzel-Kunita
Formula, now for $\beta(z)= z^2$. Then, we obtain 
$$
  \begin{aligned}    
  |w_{\varepsilon1,2}(t,x)|^2 &- |w_{\varepsilon1,2}(0,x)|^2 
  \\[5pt]
  &+ \int_{0}^{t}  \ b(s,x) \cdot \nabla w^2_{\varepsilon1,2}(s,x) \ ds \ 
	+  \int_{0}^{t} \partial_i w^2_{\varepsilon1,2}(s,x) \ {\circ}dB_{s}^{i}
\\[5pt]
 & =  2 \int_{0}^{t} w_{\varepsilon1,2} \,  \mathcal{R}_{\varepsilon1,2}(b,u) ds +
2 \int_{0}^{t} w_{\varepsilon1,2} \,  \mathcal{P}_{\varepsilon1,2}(u) \  {\circ}dB_{s}^{i}.
   \end{aligned}
$$
Then, we multiply the above equation by a test function $\varphi\in C_c^{\infty}(\R^d) $, and integrating in $U$, we obtain
 $$
  \begin{aligned}    
  &\int_U |w_{\varepsilon1,2}(t,x)|^2 \, \varphi(x) \ dx 
   -\int_U  |w_{\varepsilon1,2}(0,x)|^2 \, \varphi(x) \ dx  
 \\[5pt]
 &  -   \int_{0}^{t}\!\! \int_U  w^2_{\varepsilon1,2}(s,x)  \ b(s,x) \cdot \nabla \varphi(x)  \ dx ds 
     -   \int_{0}^{t}\!\! \int_U  w^2_{\varepsilon1,2}(s,x) \,  \dive b(s,x) \, \varphi(x)  \ dx ds 
\\[5pt]
 & - \int_{0}^{t} \!\! \int_U  w^2_{\varepsilon1,2}(s,x) \,  \partial_i \varphi(x) \ dx \circ dB_{s}^{i}
   \\[5pt]
  &+   \int_{0}^{t} \int_{\Gamma} w^2_{\varepsilon1,2}(s,r)  \;  b(s,r)  \cdot \mathbf{n}(r)  \, \varphi(r)   \ dr ds 
   +  \int_{0}^{t} \!\! \int_{\Gamma} w^2_{\varepsilon1,2}(s,r)  \;  \mathbf{n}_i(r) \varphi(r)   \ dr {\circ}  dB_{s}^{i}
\\[5pt]
& = 2 \int_{0}^{t}  \!\! \int_U  w_{\varepsilon1,2} \,  \mathcal{R}_{\varepsilon1,2}(b,u) \, \varphi(x) \, dx ds
+ 2 \int_{0}^{t}  \!\! \int_U  w_{\varepsilon1,2} \,  \mathcal{P}_{\varepsilon1,2}(u) \, \varphi(x) \ dx  \ {\circ} dB_{s}^{i}
   \end{aligned}
$$
and taking covariation with respect to $B^j$, we have for each $i= 1, \ldots, d$, 
\begin{equation}
\label{EQCOVC}
\begin{aligned}
  & [\int_U |w_{\varepsilon1,2}(t,x)|^2 \, \varphi(x) \ dx, B_i] - \int_{0}^{t} \!\! \int_U  w^2_{\varepsilon1,2}(s,x) \,  \partial_i \varphi(x) \, dx  ds
\\[5pt]
&+	\int_{0}^{t} \!\! \int_{\Gamma} w^2_{\varepsilon1,2}(s,x)  \;  \mathbf{n}_i(r) \varphi(r)   \ dr   ds= 2 \int_{0}^{t}  \!\! \int_U  w_{\varepsilon1,2} 
\,  \mathcal{P}_{\varepsilon1,2}(u) \, \varphi(x) \, dx  ds.
\end{aligned}
\end{equation}
Moreover, taking the expectation
$$
\begin{aligned}
	\int_{0}^{t} \!\! \int_{\Gamma}  &\mathbb{E} | w_{\varepsilon1,2}(s,x) |^2  \;  \mathbf{n}_i(r) \varphi(r)   \ dr   ds= -\mathbb{E}   [\int_U |w_{\varepsilon1,2}(t,x)|^2 \, \varphi(x) \ dx, B^i] 
\\[5pt]
&+  \int_{0}^{t} \!\! \int_U  \mathbb{E} |w_{\varepsilon1,2}(s,x) |^2 \,  \partial_i \varphi(x) \, dx  ds 
    + 2  \int_{0}^{t}  \!\! \int_U \mathbb{E}  [w_{\varepsilon1,2} \,  \mathcal{P}_{\varepsilon1,2}(u)] \, \varphi(x) \, dx ds,
\end{aligned}	
$$
and also $\vp(x)= \partial_i h(x)$ (see Appendix), we obtain
$$
\begin{aligned}
	\int_{0}^{t} \!\! \int_{\Gamma}  \mathbb{E} | &w_{\varepsilon1,2}(s,x) |^2   \ dr   ds= \mathbb{E}   [\int_U |w_{\varepsilon1,2}(t,x)|^2 \, \partial_i h(x) \ dx, B^i] 
\\[5pt]
&-  \int_{0}^{t} \!\! \int_U  \mathbb{E} |w_{\varepsilon1,2}(s,x) |^2 \,  \Delta h(x) \, dx  ds 
\\[5pt]
    &- 2  \sum_{i= 1}^d \int_{0}^{t}  \!\! \int_U \mathbb{E}  [w_{\varepsilon1,2} \,  \mathcal{P}_{\varepsilon1,2}(u)] \, \partial_i h(x) \, dx ds.
\end{aligned}	
$$
Since $\beta(u^\ve)$ is uniformly bounded, converges to $\beta(u)$ in $L^2([0,T] \times U \times \Omega)$, 
and $\mathcal{P}_{\varepsilon1,2}(u)$ converges to zero in $L^1$ (see similar results in \cite{DL}, and \cite{Misch}), 
it follows that 
${\{\beta(u^\ve)\}}_{\ve>0}$ is a Cauchy sequence in 
$L^2([0, T]\times  \Gamma \times \Omega)$. Then, there exists  
$\tilde{\gamma} \in L^2([0, T]\times  \Gamma \times \Omega)$,
such that  $\beta(u^\ve)$ converges 
to  $\tilde{\gamma}$ as $\ve \to 0$. 
In particular, taking $\beta(u)= u$, there exists a subsequence  of $u^\ve$,  which converges  almost sure on 
$[0,T] \times \Gamma \times \Omega$, which limit we denote by $ \gamma u$. We observe that 
$\int  \gamma u(t,r)  dr $ is adapted since is the limit of adapted process. 
	
\medskip
3.  Now, we show that $\gamma u \in L^{\infty}([0,T] \times \Gamma \times \Omega)$, and also \eqref{STOCTRACE}.  
We denote $M= \|u\|_\infty$ and consider a non-negative $\beta$ such that 
 $\beta(u)=0$ in $[-M,M]$. Multiplying  \eqref{au}
by a test function $\vp \in C^\infty_c(\R^d)$, and after integration in $U$ we obtain
 \begin{equation}
 \label{STOCTRACEREG}
  \begin{aligned}    
  &\int_U \beta(u_{\varepsilon}(t,x)) \, \varphi(x) \ dx 
   -\int_U \beta(u_{\varepsilon}(0,x)) \, \varphi(x) \ dx  
 \\[5pt]
 &  -   \int_{0}^{t}\!\! \int_U  \beta(u_{\varepsilon}(s,x))  \ b(s,x) \cdot \nabla \varphi(x)  \ dx ds 
     -   \int_{0}^{t}\!\! \int_U  \beta(u_{\varepsilon}(s,x)) \,  \dive b(s,x) \, \varphi(x)  \ dx ds 
\\[5pt]
 & - \int_{0}^{t} \!\! \int_U  \beta(u_{\varepsilon}(s,x)) \,  \partial_i \varphi(x) \,dx \circ dB_{s}^{i}
   \\[5pt]
  &+   \int_{0}^{t} \int_{\Gamma} \beta(u_{\varepsilon}(s,x))  \;  b(s,r)  \cdot \mathbf{n}(r)  \, \varphi(r)   \ dr ds 
   +  \int_{0}^{t} \!\! \int_{\Gamma} \beta(u_{\varepsilon}(t,x)) \;  \mathbf{n}_i(r) \varphi(r)   \ dr \circ  dB_{s}^{i}
\\[5pt]
& =  \int_{0}^{t}  \!\! \int_U  \beta'(u_{\varepsilon}(t,x)) \,  \mathcal{R}_{\varepsilon}(b,u) \, \varphi(x) \, dx ds
+  \int_{0}^{t}  \!\! \int_U  \beta'(u_{\varepsilon}(t,x)) \,  \mathcal{P}_{\varepsilon}(u) \, \varphi(x) \, dx  dB_{s}^{i}.
   \end{aligned}
\end{equation}
Then, we pass to the limit as $\ve \to 0$, and similarly to \eqref{EQCOVC}, we take the covariation with respect to $B^j$, to obtain 
$$
\int_{0}^{t} \!\! \int_{\Gamma} \beta(\gamma u(s,r))  \;  \mathbf{n}_i(r) \varphi(r)   \ dr   ds= 0
$$
for each $i= 1, \ldots, d$, where we have used that $\beta(u)=0$ in $[-M,M]$. 
Therefore, taking $\vp(x)= \partial_i h(x)$ and since $\beta> 0$ in $\R \setminus [-M,M]$, it follows that 
$$
   \gamma u(t,r,\omega) \in [-M,M] \quad \text{almost sure in $[0,T] \times \Gamma \times \Omega$}.
$$
Similar procedure to \eqref{STOCTRACEREG} may be establish now for any $\beta \in C^2$, and then we are allowed to 
pass to the limit as $\ve \to$ to obtain \eqref{STOCTRACE}.

\medskip
4. Finally, we show the uniqueness of the trace. If  $ \gamma_1 u$ and  $ \gamma_2 u$
are two measurable and bounded functions satisfying \eqref{STOCTRACE}, then we have
for each test function $\vp \in C_{0}^\infty(\R^d)$ and $\beta$ the identity function 
$$
\begin{aligned}
  &\int_{0}^{t} \!\! \int_{\Gamma} \gamma_1 u \, \varphi(r) \, b(s,r)  \cdot \mathbf{n}(r)  \ dr ds
 +  \int_{0}^{t} \!\! \int_{\Gamma} \gamma_1 u\, \varphi(r) \,  \mathbf{n}_i(r) \ dr \, {\circ}dB_{s}^{i}
\\[5pt]
  &=  \int_{0}^{t} \!\! \int_{\Gamma}  \gamma_2 u \, \varphi(r) \, b(s,r)  \cdot \mathbf{n}(r)  \ dr ds
  + \int_{0}^{t} \!\! \int_{\Gamma} \gamma_2 u \, \varphi(r) \,  \mathbf{n}_i(r) \ dr \, {\circ}dB_{s}^{i}.
\end{aligned}
$$
Taking the covariation with respect to $B^i$,  we obtain for each $i= 1, \ldots, d$
$$
 \int_{0}^{t} \!\! \int_{\Gamma} \gamma_1 u \, \varphi(r) \,  \mathbf{n}_i(r) \ dr \, ds =  
  \int_{0}^{t} \!\! \int_{\Gamma} \gamma_2 u \, \varphi(r) \,  \mathbf{n}_i(r) \ dr \, ds,
$$
from which follows the uniqueness of the trace, and hence the thesis of the proposition.
\end{proof}

\subsection{Weak solutions for non-regular coefficients} 
\label{STEQ}

In this section, we give the solvability of the stochastic initial-boundary value problem \eqref{trasport}
for measurable and bounded data. 
The great novelty here is the passage from Stratonovich to 
It\^o's formulation in bounded domains.

\begin{definition}
\label{defisoluIBPV} Let $u_0 \in L^\infty(U)$, 
$u_b \in L^\infty(\Gamma_T)$ be given. A stochastic process
$ u \in L^{\infty}(U_T \times \Omega)$
is called 
a  weak  $L^{\infty}-$solution of the  IBVP \eqref{trasport},
when for each test function $\varphi \in C_c^{\infty}(\R^d)$, the process $\int_{U}  u(t,
  x)\varphi(x)
  dx$ has a continuous modification which is a
$\mathcal{F}_{t}$-semimartingale, and for all $t \in [0,T]$, we have $\mathbb{P}$-almost sure
\begin{equation}
\label{IBVPEXT}
     \begin{aligned}
      \int_U  & u(t,x) \varphi(x) dx= \int_{U} u_{0}(x)\varphi(x) \ dx
      +\int_{0}^{t} \!\! \int_{U}u(s,x) \, b^j(s,x)  \, \partial_j \varphi(x) \ dx ds
\\      
        &+ \int_{0}^{t}\!\!  \int_{U} u(s,x) \, \dive \, b(s,x)\,  \varphi(x) \ dx ds 
           -  \int_{0}^{t} \!\! \int_{\Gamma} u_\mathbf{o}(s,r) \mathbf{n^o}_{\!\!j} \ b^j(s,r) \ \varphi(r) \, dr ds
\\[5pt]
     &+ \int_{0}^{t} \!\! \int_{\Gamma} u_b(s,r) \mathbf{n^i}_{\!\!j} \, b^j(s,r) \ \varphi(r)  \ dr ds
     -\int_{0}^{t} \!\! \int_{\Gamma}  u_\mathbf{o}(s,r) \mathbf{n^o}_{\!\!j} \, \varphi(r)  \ d{r}\,  {\circ}dB_{s}^{j}
\\[5pt]
    &+\int_{0}^{t} \int_{\Gamma} u_b(s,r) \ \mathbf{n^i}_{\!\!j} \, \varphi(r)  \ dr \circ
        dB_{s}^{j} 
        + \int_{0}^{t} \!\! \int_{U} u(s,x)  \; \partial_j \varphi(x) \ dx \, {\circ}dB_{s}^{j}.
   \end{aligned}
\end{equation}
\end{definition}

\begin{remark}
\label {stochastic influx}
Clearly, the term $\mathbf{n^i}$ (also $\mathbf{n^o}$) in \eqref{IBVPEXT} should be explained, 
since the stochastic influx boundary zone, that is 
$\Gamma^{\rm{in}}(\omega)$ was established for regular drift vector field $b$.  
Indeed, under the assumption that $b$ satisfies \eqref{con1}, \eqref{BCONDUNQ},
we may follow Fedrizzi, Flandoli see \cite{Fre1,Fre2}, and show
the $\alpha$-H\"{o}lder continuity of the stochastic flow  $X_{s,t}$,
for each $\alpha \in  (0, 1) $. Also that, it is a stochastic flow of homeomorphism.
Then, we may consider the inverse  $Y_{s,t}:=X_{s,t}^{-1}$, and
define $\kappa> 0$, $\bar{Y}$,  
and $\Gamma^{\rm{in}}(\omega)$
as introduced in Section \ref{preliminares}.
\end{remark}

For convenience we extend the weak solution $u \in L^{\infty}(U_T \times \Omega)$ by setting 
$$
   u(t,x,\omega) \equiv 0, \quad \text{for all $(t,x,\omega) \in 
   (\R \times U \times \Omega) \setminus (U_T \times \Omega)$}.
$$
Then, we consider the following main general existence result.
\begin{theorem}
\label{theoExis} Under condition \eqref{con1}, \eqref{BCONDUNQ}, 
there exits a weak $L^{\infty}-$solution $u \in L^{\infty}(U_T \times \Omega)$ of the IBVP 
\eqref{trasport}.
\end{theorem}

\begin{proof} 
1. For each $\ve> 0$, let us denote by $u_0^\ve$, $u_b^\ve$ respectively the
standard mollifications of $u_0$ and $u_b$, satisfying compatibility conditions. 
Similarly, $b^{\ve}$ the mollification of $b$.
Let   $X_t^{\ve }$  be  the associated flow 
given by \eqref{itoass}, and define  (see Lemma \ref{lemaexis})
\begin{equation}\label{reprereg}
u^\ve(t,x):=
\left \{
\begin{aligned}
       {u}_{0}^\ve(\bar{Y}^\ve_{\kappa^\ve,t}(x)), & \quad \text{if  $\kappa^\ve(t,x)= 0$},
\\[5pt]
       {u}_{b}^\ve(\kappa^\ve, \bar{Y}^\ve_{\kappa^\ve,t}(x)), & \quad \text{if  $\kappa^\ve(t,x)> 0$},
\end{aligned}
\right.
\end{equation}
where $\kappa^\ve> 0 $ is given by \eqref{TAU}. 
Thus  $u^{\ve }(t,x)$ is uniformly bounded, 
with respect to $\ve >0$, and satisfies for each test function
$\varphi \in C^\infty_c(\R^d)$ and all $t \in [0,T]$
\begin{equation}
\label{IBVPEXT1}
     \begin{aligned}
      \int_U  & u^\ve(t,x) \varphi(x) \ dx= \int_{U} u^\ve_{0}(x)\varphi(x) \ dx
      +\int_{0}^{t} \!\! \int_{U} u^\ve(s,x) \, b^\ve(s,x) \cdot \nabla \varphi(x) \ dx ds
\\      
        &+ \int_{0}^{t}\!\!  \int_{U} u^\ve(s,x) \, \dive b^\ve(s,x) \, \varphi(x) \ dx ds 
           -  \int_{0}^{t} \!\! \int_{\Gamma} \gamma u^\ve(s,r) \, \varphi(r) \, b^\ve \cdot \mathbf{n}  \, dr ds
\\[5pt]
     &-\int_{0}^{t} \!\! \int_{\Gamma}  \gamma u^\ve(s,r) \, \varphi(r)  \, \mathbf{n}_{j} \ d{r}\,  {\circ}dB_{s}^{j}
      + \int_{0}^{t} \!\! \int_{U} u^\ve(s,x) \ \partial_j \varphi(x) \ dx \, {\circ}dB_{s}^{j},
   \end{aligned}
\end{equation}
where $ \gamma u^\ve \, \mathbf{n}= u^\ve_\mathbf{o} \, \mathbf{n^o}_{\!\!\!\ve} - u^\ve_b \, \mathbf{n^i}_{\!\!\ve}$ almost sure,
with $\mathbf{n^i}_{\!\!\ve}= (-  \chi_{\{\bar{Y}^\ve / \kappa^\ve> 0 \}}) \, \mathbf{n}$, and 
analogously $\mathbf{n^o}_{\!\!\!\ve}$.

\underline{Claim 1}: The family $\{\kappa^\ve\}_{\ve>0}$ converges to $\kappa$ as $\ve \to 0$ for almost all $\omega \in \Omega$, and
a.e. $(t,x) \in U_T$. 

Proof of Claim 1: Indeed, due to Remark \ref{stochastic influx} for a.e. $(t,x) \in U_T$ fixed,
we have that $Y^\ve(s)$ converges to $Y(s)$ as $\ve \to 0$
uniformly over any closed interval in $[0,T]$, where $Y^\ve(s) \equiv Y^\ve(s,t,x)$, similarly $Y(s)$. 
Hence given $\eta> 0$, there exists a $\ve_0> 0$ (which does not depend on $s$), such that 
if $0< \ve < \ve_0$, then 
\begin{equation}
\label{TAUCONV}
    |Y^\ve(s) - Y(s)|< \eta,
\end{equation}
which is to say, there exists a tubular neighborhood $\pi$ around $Y(s)$
with radius $\eta>0$, such that, $(s,Y^\ve(s)) \subset \pi$ for any $s \in (0,t)$, and $\ve < \ve_0$. 
Now, since the domain $U$ has regular boundary, upon rotating and relabeling the coordinates 
axes if necessary, we may locally represent the lateral boundary $[0,T] \times \Gamma$ by a graph, say 
$\Pi$. By definition, let $\kappa> 0$ be the first value 
of $s$ such that, $Y(s) \in \Pi$. 
Analogously, $\kappa^\ve> 0$, such that $Y^\ve(\kappa^\ve) \in \pi \cap \Pi$. Therefore, from \eqref{TAUCONV} 
$$
   \lim_{\ve \to 0} \kappa^\ve(\omega)= \kappa(\omega).
$$ 
This is clear (because of the uniform convergence) when the curve $Y(s)$ cross $\Pi$, but 
it may happens that $Y(s)$ touch $\Pi$ at some point, say $Y(\kappa)$, 
and the above convergence become false. 
Although, due to Brownian motion's property, the set 
of points where $Y(s)$ is tangent to $\Pi$
has probability zero (recall that the Brownian motion is 
nowhere differentiable almost surely).

2. Now, we conveniently reformulate equation \eqref{IBVPEXT1} 
in It\^o's form. One remarks that,
differently from Remark \ref{lemmaito1} we have to deal with boundary terms,
which have never been done before in the literature.
To begin, let us consider the relation between It\^o and Stratonovich integrals 
in  \eqref{IBVPEXT1}, that is
\begin{equation}
\label{Q1}
   \int_{0}^{t} \!\! \int_{U} u^\ve(s,x)  \; \partial_{j} \varphi(x) \ dx \, {\circ}dB_{s}^{j}   
   = \int_{0}^{t} \!\! \int_{U} u^\ve(s,x)  \; \partial_{j} \varphi(x) \ dx \ dB_{s}^{j}
   +\frac{I_1}{2},
\end{equation}
\begin{equation}
\label{Q2}
       \int_{0}^{t} \!\! \int_{\Gamma}  \gamma u^\ve(s,r) \, \varphi(r)  \, \mathbf{n}_j(r) \ d{r}\,  {\circ}dB_{s}^{j}
       = \int_{0}^{t} \!\! \int_{\Gamma} \gamma u^\ve(s,r) \, \varphi(r)  \, \mathbf{n}_j(r) \ d{r}\, dB_{s}^{j}  
       +\frac{I_2}{2},
\end{equation}
where 
$$
    I_1:=  \left[    \int_{U} u^\ve(.,x)  \; \partial_{j} \varphi(x) \ dx,  B_{(.)}^{j} \right]_t, 
    \quad I_2:= \left[     \int_{\Gamma} \gamma u^\ve(.,r) \, \varphi(r)  \, \mathbf{n}_j(r) \ d{r},  B_{(.)}^{j} \right]_t,
$$
and $[., .]_t$ denotes the joint quadratic variation, 
which is a bounded variation term (see Appendix
for more details). In fact, we 
compute these two joint quadratic variations
above, from equation \eqref{IBVPEXT1} with special test functions, and 
observe that, only the martingale part 
have to be considered. 

To compute $I_1$, we replace $\vp$ in  \eqref{IBVPEXT1}--\eqref{Q2} by $\partial_j \vp$. Then, for each $j= 1, \ldots, d$, 
the martingale part of  $\int_{U} u^\ve(t,x)  \; \partial_{j} \varphi(x) \ dx$  is 
$$
   \int_{0}^{t} \!\! \int_{U} u^\ve(s,x)  \;  \partial_{i}\big( \partial_{j} \varphi(x)\big) \ dx \, dB_{s}^{i}
   - \int_{0}^{t} \!\! \int_{\Gamma} \gamma u^\ve(s,r) \, \partial_{j} \varphi(r)  \, \mathbf{n}_i(r) \ d{r}\,  dB_{s}^{i }.
$$
Thus, we have  
\begin{equation}\label{m1}
          I_1=  \int_{0}^{t}  \int_{U} u^\ve(s,x)  \;  \partial_{j}^{2} \varphi(x) \ dx \, ds
          - \int_{0}^{t}  \int_{\Gamma} \gamma u^\ve(s,r) \,  \partial_{j}  \varphi(r)  \, \mathbf{n}_j(r) \ d{r} ds.
\end{equation}
Now, we compute $I_2$. Similarly, we replace $\vp(x)$  
in equation \eqref{IBVPEXT1}--\eqref{Q2} by $\varphi(x)  \, \partial_{j}\zeta_\mu (h(x))$ as a test function, where 
for $\mu> 0$, $\zeta_\mu: \R \to [-1,1]$ is given by 
$$
   \zeta_\mu(\tau):= 
   \left \{
   \begin{aligned}
  \sgn \tau, \quad &\text{if $|\tau| > \mu$}, 
  \\[5pt]
  \frac{\tau}{\mu} \hspace{5pt}, \quad &\text{if $|\tau| \leq \mu$},
   \end{aligned}
   \right.
$$ 
with $h(x)$ the given function at the Appendix. 
Certainly, we have to mollify $\zeta_\mu$ by a standard mollifier $\rho_n$ to have the necessary regularity,
and then first pass to the limit as $n \to \infty$ (we omit this standard procedure). To begin, we consider
the left hand side of  \eqref{IBVPEXT1}, then we pass to the martingale terms in the right hand side of it.
 
 \medskip
{\underline{Claim 2}}: For each $t \in [0,T]$, and $j= 1,\ldots,d$, it follows that
\begin{equation}\label{areaeycoarea}
    \ess \lim_{\mu \rightarrow 0^+ }\int_U \, u^\ve(t,x) \; \varphi(x) \,  \partial_{j} \zeta_\mu(h(x)) dx=  
     -\int_{\Gamma}  \gamma u^\ve(t,r) \, \varphi(r)  \, \mathbf{n}_j(r) \ d{r},
\end{equation}
for each test function $\varphi \in C^\infty_c(\R^d)$. 

Proof of Claim 2: Fix any point $r \in \Gamma$. Then, since $\Gamma$ is $C^2$, there exists a neighbourhood 
$W$ of $r$ in $\R^d$, an open set $V \subset \R^{d-1}$ and 
a $C^2$ mapping $\zeta: V \to \Gamma \cap W$, which is a $C^1-$diffeomorphism. 
Let $\Psi_\tau:[0,1] \times \Gamma \to \ol{U}$ be an admissible deformation (see Appendix), and recall that
$$
   \lim_{\tau \to 0} J[\Psi_\tau \circ \zeta]= J[\zeta] \quad \text{in $C(V)$},
$$
where $J\Psi_\tau$ denotes the Jacobian of the map $\Psi_\tau$.
Now, we set $\Upsilon= \Gamma \cap W$, $\Upsilon^\tau= \Psi_\tau(\Upsilon)$,
and consider $\varphi \in \clg{E}$, where $\clg{E}$ is a countable dense subset of $C_c^\infty(W)$.
Therefore, applying the Coarea Formula for the function $h$,  we have 
for each $t \in [0,T]$ and $\mu> 0$ (sufficiently small) 
\begin{equation}
\label{AUX1}
   \begin{aligned}
   \int_{U \cap W} \ u^\ve(t,x) \, \varphi(x)  \, \partial_{j} \zeta_\mu(h(x)) \, dx
   &= -\int_{0}^{\mu} \!\!\! \int_{\Upsilon^{\tau}} u^\ve(t,r) \; \varphi(r) \;  \zeta'_\mu(\tau) \ \mathbf{n}^\tau_j(r) \, dr d\tau
   \\[5pt]
   &= -\frac{1}{\mu} \int_{0}^{\mu} \!\!\! \int_{\Upsilon^{\tau}} u^\ve(t,r) \; \varphi(r)  \ \mathbf{n}^\tau_j(r) \, dr d\tau.
\end{aligned}
\end{equation}
The goal now is to
pass to the limit as $\mu \to 0^+$, consequently as $\tau \to 0^+$.  
First we apply the Area Formula for $\Psi_\tau$
in the right hand side of the above equation. 
Indeed, observing that we may replace $\varphi(r)$ by $\varphi(\Psi^{-1}_\tau(r))$, 
similarly $\mathbf{n}^\tau(r)$ by $\mathbf{n}(\Psi^{-1}_\tau(r))$,
with an error that goes to zero as $\tau \to 0^+$, we have
\begin{equation}
\label{AUX2}
\begin{aligned}
   \int_{0}^{\mu} \!\!\! \int_{\Upsilon^{\tau}}  u^\ve(t,r) \;& \varphi(\Psi^{-1}_\tau(r)) \ \mathbf{n}_j(\Psi^{-1}_\tau(r)) \, dr d\tau
   \\[5pt]
   &=  \int_{0}^{\mu} \!\!\! \int_{\Upsilon} u^\ve(t,\Psi_\tau(r)) \; \varphi(r)  \ \mathbf{n}_j(r) \ J[\Psi_\tau] \, dr d\tau,
\end{aligned}
\end{equation}
where $J[\Psi_\tau]$ is defined by 
$$
J[\Psi_\tau](r):= \dfrac{J[\Psi_\tau\circ\zeta](\zeta^{-1}(r))}{J[\zeta](\zeta^{-1}(r))},
$$
and satisfies $J[\Psi_\tau]\to 1$ uniformly as $\tau \to 0$. 
Passing to the limit as $\mu \to 0$, we obtain from \eqref{AUX1}, \eqref{AUX2} 
$$
   \begin{aligned}
   \ess \lim_{\mu \rightarrow 0^+} & \int_{U \cap W} \, u^\ve(t,x) \; \varphi(x) \,  \partial_{j} \zeta_\mu(h(x)) dx
   \\[5pt]
     &= -\ess \lim_{\mu \rightarrow 0^+} \Big( \frac{1}{\mu} \int_{0}^{\mu}\!\!\!  \int_{\Upsilon} u^\ve(t,\Psi_\tau(r)) 
     \; \varphi(r) \ \mathbf{n}_j(r) \ J[\Psi_\tau] \, dr d\tau \Big)
     \\[5pt]
     &=   -\int_{\Upsilon}  \gamma u^\ve(t,r) \, \varphi(r)  \, \mathbf{n}_j(r) \ d{r}
   \end{aligned}
$$
for each test function $\varphi \in C^\infty_c(W)$, 
where we used the density of $\clg{E}$ in $C^\infty_c(W)$,
the Dominated Convergence Theorem, 
and Remark \ref{REMTRACEREGCOEFF}.

Finally, since $\Gamma$ is a compact set, we can applying a standard partition of unity
argument, exchange $U \cap W$, $\Upsilon$ respectively by $U$,  
$\Gamma$ in the previous steps, which is to say, 
consider the general case. So the claim is proved.

Henceforth, this standard procedure of partition of unity, applied above, 
is considered implicitly.

Now, let us study for $j=1, \ldots, d$, 
$$
\begin{aligned}
   \int_{0}^{t} \!\! \int_{U} u^\ve(s,x)  &\; \partial_{i} \Big(\varphi(x) \, \partial_j \zeta_\mu(h(x))\Big)  \ dx dB_{s}^{i}  
   \\[5pt]
      &- \int_{0}^{t} \!\! \int_{\Gamma}  \gamma u^\ve(s,r) \, \varphi(r)  
       \, \partial_j \zeta_\mu(h(r)) \, \mathbf{n}_i(r) \ d{r}\, dB_{s}^{i},
\end{aligned}
$$
or after some computations 
$$
\begin{aligned}
   \int_{0}^{t} \!\! \int_{U} u^\ve(s,x)  &\; \partial_{i} \varphi(x) \, \zeta'_\mu(h(x)) \, \partial_j h(x) \ dx dB_{s}^{i}  
   \\[5pt]
   &+  \int_{0}^{t} \!\! \int_{U} u^\ve(s,x)  \; \varphi(x) \, \zeta'_\mu(h(x)) \,\partial_{i} \partial_j h(x) \ dx dB_{s}^{i}  
   \\[5pt]
   &+  \int_{0}^{t} \!\! \int_{U} u^\ve(s,x)  \; \varphi(x) \, \zeta''_\mu(h(x)) \,\partial_{i} h(x) \, \partial_j h(x) \ dx dB_{s}^{i}  
   \\[5pt]
      &- \int_{0}^{t} \!\! \int_{\Gamma}  \gamma u^\ve(s,r) \, \varphi(r)  
       \, \zeta_\mu'(0)  \, \partial_j (h(r)) \, \mathbf{n}_i(r) \ d{r}\, dB_{s}^{i}.
\end{aligned}
$$
Therefore, taking the variation in the above terms, we obtain 
\begin{equation} 
\label{J1}
\begin{aligned}
   \int_{0}^{t} \!\! \int_{U} u^\ve(s,x)  &\; \partial_{i} \varphi(x) \, \zeta'_\mu(h(x)) \, \partial_i h(x) \ dx ds  
   \\[5pt]
   &+  \int_{0}^{t} \!\! \int_{U} u^\ve(s,x)  \; \varphi(x) \, \zeta'_\mu(h(x)) \,\partial^2_{i} h(x) \ dx ds
   \\[5pt]
   &+  \int_{0}^{t} \!\! \int_{U} u^\ve(s,x)  \; \varphi(x) \, \zeta''_\mu(h(x)) \,|\partial_{i} h(x)|^2 \ dx ds
   \\[5pt]
      &- \int_{0}^{t} \!\! \int_{\Gamma}  \gamma u^\ve(s,r) \, \varphi(r)  
       \, \zeta_\mu'(0)  \, \partial_i h(r) \, \mathbf{n}_i(r) \ d{r}\, ds
       \\[5pt]
       &=: J_1 + J_2 + J_3-J_4,
   \end{aligned}
\end{equation}
with obvious notations. 

{\underline{Claim 3}}: For each $t \in [0,T]$, 
and all test functions $\varphi \in C^\infty_c(\R^d)$,
it follows that:
\begin{equation}
\label{areaeycoarea1}
\begin{aligned}
i)& \ess \lim_{\mu \rightarrow 0^+ }  J_1=  
     -\int_0^t \!\! \int_{\Gamma}  \gamma u^\ve(s,r) \, \nabla \varphi(r)  \cdot \mathbf{n}(r) \ d{r} ds,
\\[5pt]     
ii)& \ess \lim_{\mu \rightarrow 0^+ }  J_2= (d-1)  \int_0^t \!\! \int_{\Gamma}  \gamma u^\ve(s,r) \, \varphi(r)  \, H(r) \ d{r} ds,
\\[5pt]
iii)& \ess \lim_{\mu \rightarrow 0^+ }  (J_3 - J_4)=  0,
\end{aligned}
\end{equation}
where $H$ is the mean curvature of $\Gamma$. 

Proof of Claim 3: Assertion $(i)$ and $(ii)$ follow similarly to the proof of Claim 2. 
Thus, let us show item $(iii)$. Moreover, as mentioned before we omit the localization 
procedure and the partition of unit argument. 
Applying the Coarea Formula for the function 
$h$, and then the Area Formula for the map $\Psi_\tau$,
we have
$$
   \begin{aligned}
& \ess \!\! \lim_{\mu \to 0^+} (J_3 - J_4)
\\[5pt]
&= \ess \!\! \lim_{\mu \to 0^+} \Big(- \frac{1}{\mu} \int_0^t \!\!\! \int_{U} u^\ve(s,x)  \; \varphi(x) \ \delta_\mu(h(x)) 
   \, |\partial_i h(x)|^2 \, dx ds
\\[5pt]   
    &\qquad + \frac{1}{\mu} \int_0^t \!\!\! \int_{\Gamma} \gamma u^\ve(s,r)  \; \varphi(r) \,
    \, |\partial_i h(r)| \, dr ds \Big)
    \\[8pt]
    &= \ess \!\! \lim_{\mu \to 0^+} \Big(- \frac{1}{\mu} \int_0^t \!\!\! \int_0^\mu \!\!\! \int_{\Gamma^\tau} 
    u^\ve(s,r)  \; \varphi(r) \ \delta_\mu(\tau) \, |\partial_i h(r)| \, dr d\tau ds
\\[5pt]   
    &\qquad + \frac{1}{\mu} \int_0^t \!\!\! \int_0^\mu \!\!\! \int_{\Gamma} u^\ve(s,\Psi_\tau(r))  \; \varphi(\Psi_\tau(r)) \,
    \delta_0(\tau)
    \, |\partial_i h(\Psi_\tau(r))| J[\Psi_\tau] \, dr d\tau ds \Big)
    \\[8pt]
    &= \ess \!\! \lim_{\mu \to 0^+} \frac{1}{\mu} \int_0^t \!\!\! \int_0^\mu \!\!\! \int_{\Gamma} 
    u^\ve(s,\Psi_\tau(r))  \; \varphi(\Psi_\tau(r)) \ (\delta_0-\delta_\mu) \, |\partial_i h|  J[\Psi_\tau] \, dr d\tau ds= 0,
\end{aligned}
$$
where $\delta_\mu$ is the (approaching sequence) Dirac measure concentrated at $\mu$,
and we have used the Dominated Convergence Theorem. Therefore, the proof of 
Claim 3 is finished.

We are ready to write equation \eqref{IBVPEXT1} 
in the equivalent It\^o's form (bounded domains), that is
\begin{equation}
\label{IBVPEXTIto}
     \begin{aligned}
      \int_U  & u^\ve(t,x) \varphi(x) dx= \int_{U} u^\ve_{0}(x) \varphi(x) \ dx
      +\int_{0}^{t} \!\! \int_{U} u^\ve(s,x) \, b^\ve(s,x)  \cdot \nabla \varphi(x) \ dx ds
\\      
        &+ \int_{0}^{t}\!\!  \int_{U} u^\ve(s,x) \, \dive \, b^\ve(s,x)\,  \varphi(x) \ dx ds 
           -  \int_{0}^{t} \!\! \int_{\Gamma} \gamma u^\ve(s,r) \, \varphi(r) \, b^\ve \cdot \mathbf{n} \ dr ds
\\[5pt]
     &-\int_{0}^{t} \!\! \int_{\Gamma}  \gamma u^\ve(s,r) \, \varphi(r)  \, \mathbf{n}_j(r) \ d{r} dB_{s}^{j}
    -  \int_{0}^{t} \!\! \int_{\Gamma}  \gamma u^\ve(s,r) \, \nabla \varphi(r) \cdot \mathbf{n}(r) \, dr ds
\\[5pt]
  &+ \frac{(d-1)}{2}   \int_{0}^{t} \!\! \int_{\Gamma}  \gamma u^\ve(s,r) \, \varphi(r) \, H(r) \, dr ds 
\\[5pt]
    &+ \int_{0}^{t} \!\! \int_{U} u^\ve(s,x)  \; \partial_j \varphi(x) \ dx dB_{s}^{j} 
    + \frac{1}{2} \int_{0}^{t} \!\! \int_{U} u^\ve(s,x) \, \Delta \varphi(x) \,  dx ds.
   \end{aligned}
\end{equation}

4. Limit transition. Since
the family $\{u^{\varepsilon}\}$ by our construction given by Lemma \ref{lemaexis}
is uniformly bounded up to the boundary, there exists
a function $u \in L^{\infty}(U_T \times \Omega)$, 
the weak-star limit of $u^\ve$ as $\ve \to \infty$,
such that the process $ \int_U u(t,x) \varphi(x) dx$ is adapted, 
since it is the weak limit in $L^{2}([0, T] \times \Omega)$ of adapted processes, see  \cite{Pardoux} Chapter III.
Analogously, there exists a function $u_\Gamma \in L^\infty([0,T] \times \Gamma \times \Omega)$, which is 
the weak-star limit of $\gamma u^\ve$, such that the process $ \int_\Gamma u_\Gamma(t,r) \varphi(r) dr$ is adapted,
since, passing to the limit as as $\ve \to 0$ in \eqref{IBVPEXTIto}, we have
$$
     \begin{aligned}
      \int_U  & u(t,x) \varphi(x) dx= \int_{U} u_{0}(x) \varphi(x) \ dx
      +\int_{0}^{t} \!\! \int_{U} u(s,x) \, b(s,x) \cdot \nabla \varphi(x) \ dx ds
\\      
        &+ \int_{0}^{t}\!\!  \int_{U} u(s,x) \, \dive \, b(s,x)\,  \varphi(x) \ dx ds 
           -  \int_{0}^{t} \!\! \int_{\Gamma} u_\Gamma(s,r) \, \varphi(r) \, b \cdot \mathbf{n} \ dr ds
\\[5pt]
     &-\int_{0}^{t} \!\! \int_{\Gamma}  u_\Gamma(s,r) \, \varphi(r)  \, \mathbf{n}_j(r) \ d{r} dB_{s}^{j}
    -  \int_{0}^{t} \!\! \int_{\Gamma}  u_\Gamma(s,r) \, \nabla \varphi(r) \cdot \mathbf{n}(r) \, dr ds
\\[5pt]
  &+ \frac{(d-1)}{2}   \int_{0}^{t} \!\! \int_{\Gamma}  u_\Gamma(s,r) \, \varphi(r) \, H(r) \, dr ds 
\\[5pt]
    &+ \int_{0}^{t} \!\! \int_{U} u(s,x)  \; \partial_j \varphi(x) \ dx dB_{s}^{j} 
    + \frac{1}{2} \int_{0}^{t} \!\! \int_{U} u(s,x) \, \Delta \varphi(x) \,  dx ds,
   \end{aligned}
$$
or equivalently 
\begin{equation}
\label{IBVPEXT10}
     \begin{aligned}
      \int_U  & u(t,x) \varphi(x) dx= \int_{U} u_{0}(x)\varphi(x) \ dx
      +\int_{0}^{t} \!\! \int_{U}u(s,x) \, b^j(s,x)  \, \partial_j \varphi(x) \ dx ds
\\      
        &+ \int_{0}^{t}\!\!  \int_{U} u(s,x) \, \dive \, b(s,x)\,  \varphi(x) \ dx ds 
           -  \int_{0}^{t} \!\! \int_{\Gamma} u_\Gamma(s,r) \ b \cdot \mathbf{n} \ \varphi(r) \, dr ds
\\[5pt]
     &-\int_{0}^{t} \!\! \int_{\Gamma}  u_\Gamma(s,r) \mathbf{n}_{j} \, \varphi(r)  \ d{r}\,  {\circ}dB_{s}^{j}
        + \int_{0}^{t} \!\! \int_{U} u(s,x)  \; \partial_j \varphi(x) \ dx \, {\circ}dB_{s}^{j}.
   \end{aligned}
\end{equation}

5. Finally we show \eqref{IBVPEXT}. First, 
we observe that $u$ is also a distributional $L^{\infty}-$solution of \eqref{trasport}. 
Then, from equation \eqref{STOCTRACE} with $\beta(z)= z$ 
and equation \eqref{IBVPEXT10}, we have 
\begin{equation}
\label{trace}
     \begin{aligned}
           &\int_{0}^{t} \!\! \int_{\Gamma} \gamma u(s,r) \, \varphi(r)  \ b \cdot \mathbf{n} \ drds
     +\int_{0}^{t} \!\! \int_{\Gamma}  \gamma u(s,r) \, \varphi(r)  \, \mathbf{n}_j \ d{r} \,  {\circ}dB_{s}^{j}
			 \\[5pt]
		 &=\int_{0}^{t} \!\! \int_{\Gamma} u_\Gamma(s,r) \, \varphi(r) \, b \cdot \mathbf{n} \ drds
     + \int_{0}^{t} \!\! \int_{\Gamma}  u_\Gamma(s,r) \, \varphi(r)  \, \mathbf{n}_j \ d{r}\,  {\circ}dB_{s}^{j}.
   \end{aligned}
\end{equation}
Therefore, taking covariation with respect to $B^j$, we obtain for $j= 1,\ldots,d$
$$
     \int_{0}^{t} \!\! \int_{\Gamma}  \gamma u(s,r) \, \varphi(r)  \, \mathbf{n}_j \ d{r}\,  ds=
     \int_{0}^{t} \!\! \int_{\Gamma}   u_\Gamma(s,r) \, \varphi(r)  \, \mathbf{n}_j \ d{r}\,  ds, 
$$
which is to say, $\gamma u = u_\Gamma$ almost sure. 
Consequently, from the uniqueness of the limit and Claim 1, it follows that
$ \gamma u \, \mathbf{n}= u_\mathbf{o} \, \mathbf{n^o} - u_b \, \mathbf{n^i}$, which
shows \eqref{IBVPEXT}, and the theorem is proved.
\end{proof}

\section{Uniqueness} \label{UNIQUENESS}

In this section, we present the uniqueness theorem
for the SPDE (\ref{trasport}).  We prove uniqueness following 
the concept of renormalized solutions introduced  by DiPerna,
Lions.  The BV framework is the one adopted in
the sequel, where we make extensive use of  the ideas from \cite{ambrosio}.

\begin{lemma} \label{leuni}  
Assume condition \eqref{con1}. Let $u$ be a distributional 
$L^\infty$-solution of the stochastic IBVP \eqref{trasport}, and define  $v:= \mathbb{E}(\beta(u))$
for any $\beta \in C^2(\R)$. 
Then, for each $u_{0} \in L^{\infty}(U)$ the function $v(t,x)$ satisfies
\begin{equation}
\label{norma}
      \partial_t v(t,x) + b(t,x) \cdot \nabla v(t,x)= \frac{1}{2} \Delta  v(t,x)
     \quad \text{in   $ \mathcal{D}'([0,T) \times U)$}.   
\end{equation}
\end{lemma}

\begin{proof} 
1. For $\ve> 0$, we define $U_\ve:= \{x \in U : \rm{dist}(x, \partial U) > \ve \}$. 
Let $\phi_{\varepsilon}$ be a standard symmetric mollifier (with support
on a ball of radius less than $\ve$), and 
$u$  a distributional  $L^{\infty}$-solution of (\ref{trasport}). Then,
for each $t \in [0,T]$, and $x \in U_\ve$ taking $\phi_\ve$ as a test function in 
\eqref{DISTINTSTR}, we obtain
$$
   \begin{aligned}
u_{\varepsilon}(t,x)  \equiv &\int_U  u(t,z) \phi_{\ve}(x-z) dz= u_0\ast \phi_{\varepsilon}(x) 
\\[5pt]
&+ \int_{0}^{t} \int_U u(s,z) \, b^i(s,z) \, \partial_i \phi_{\ve}(x-z)  \ dz ds 
\\[5pt]
&+ \int_{0}^{t} \int_U u(s,z) \, \dive b(s,z) \, \phi_{\ve}(x-z) \ dz ds 
\\[5pt]
&+  \int_{0}^{t} \int_U
 u(s,z)\,  \partial_i \phi_{\ve}(x-z) \, dz \circ dB_{s}^{i}.
\end{aligned}
$$
For $\beta \in C^{2}(\mathbb{R})$, we apply the  It\^o-Ventzel-Kunita
formula (see Theorem 8.3 of \cite{Ku2}
in the above equation,
hence we have   
\begin{equation}
\label{EQPARABOLIC}
   \begin{aligned}
   \beta(u_{\varepsilon}(t,x))&= \beta(u_0\ast \phi_{\varepsilon}(x)) 
\\[5pt]
&+ \int_{0}^{t}  \beta'(u_{\varepsilon}(s,x))  \int_U u(s,z) \, b^i(s,z) \, \partial_i \phi_{\ve}(x-z)  \ dz ds 
\\[5pt]
&+ \int_{0}^{t} \beta'(u_{\varepsilon}(s,x)) \int_U u(s,z) \, \dive b(s,z) \, \phi_{\ve}(x-z) \ dz ds 
\\[5pt]
&+  \int_{0}^{t} \beta'(u_{\varepsilon}(s,x)) \int_U
 u(s,z)\,  \partial_i \phi_{\ve}(x-z) \, dz \circ dB_{s}^{i}.
\end{aligned}
\end{equation}

2.  Now it becomes clear our strategy, which follows the renormalization procedure. Then, 
we take a test function $\vp \in C^\infty_c(U)$, multiply equation \eqref{EQPARABOLIC}
by it, and integrate in $U$ to obtain 
$$
   \begin{aligned}
   \int_U &  \beta(u_{\varepsilon}(t)) \, \vp \,dx = \int_U \beta(u_0\ast \phi_{\varepsilon}(x)) \, \vp(x) \ dx 
\\[5pt]
&+ \int_{0}^{t} \int_U \int_U \beta'(u_{\varepsilon}(s,x)) u(s,z) \, b^i(s,z) \, \partial_i \phi_{\ve}(x-z) \, \vp(x)  \ dz dx  ds 
\\[5pt]
&+ \int_{0}^{t} \int_U \int_U \beta'(u_{\varepsilon}(s,x)) u(s,z) \, \dive b(s,z) \, \phi_{\ve}(x-z) \, \vp(x) \ dz dx ds 
\\[5pt]
&+  \int_{0}^{t} \int_U \int_U \beta'(u_{\varepsilon}(s,x)) 
 u(s,z)\,  \partial_i \phi_{\ve}(x-z) \, \vp(x) \, dz dx \circ dB_{s}^{i}, 
\end{aligned}
$$
where we have used Fubini's Stochastic Theorem, see for instance \cite{EPP}. 
Since $\phi_\ve$ is a symmetric mollifier, from an algebraic 
convenient manipulation and integration by parts, we obtain
\begin{equation}
\label{EQPARABOLIC10}
   \begin{aligned}
   \int_U &  \beta(u_{\varepsilon}(t)) \, \vp \,dx - \int_U \beta(u_0\ast \phi_{\varepsilon}(x)) \, \vp(x) \ dx 
\\[5pt]
&- \int_{0}^{t} \int_U  \beta(u_{\varepsilon}(s,x)) \, b^i(s,x) \, \partial_i  \vp(x)  \ dx  ds 
\\[5pt]
&- \int_{0}^{t} \int_U  \beta(u_{\varepsilon}(s,x)) \, \dive b(s,x) \, \vp(x) \ dx ds 
\\[5pt]
&-  \int_{0}^{t} \!\! \int_U  \beta(u_{\varepsilon}(s,x)) \, \partial_i \vp(x) \, dx \circ dB_{s}^{i}
      =      \int_{0}^{t}\!\!  \int_U \beta^\prime(u_\ve(s,x)) \varphi(x)  \mathcal{R}_{\ve}(b,u)  \ dx ds,
\end{aligned}
\end{equation}
where $ \mathcal{R}_{\ve}(b,u)$ is the commutator defined as
$$
    \mathcal{R}_{\ve}(b,u)=(b\nabla ) (\phi_{\ve}\ast u )- \phi_{\ve}\ast((b\nabla)u).
$$ 
One remarks that, the commutator above is similar 
to that one used by DiPerna, Lions in \cite{DL}.
Moreover, by the regularity assumptions on $b$ and $u$, applying the Commuting Lemma 
(see \cite{ambrosio} or Theorem 9 of \cite{Falnanas}), it follows that 
$$
 \lim_{\ve \to 0} \mathcal{R}_{\ve}(b,u)= 0, \quad \mathbb{P} \ a.s \mbox{ in }\ L^{1}([0,T];
L^{1}_{\loc}(\mathbb{R}^{d})).
$$
Therefore, since $u$ is measurable and bounded, $u_\ve$ converges to $u$ in $L^1_{\loc}$, 
we obtain from \eqref{EQPARABOLIC10} 
passing to the limit as $\ve \to 0$ 
\begin{equation}
\label{EQPARABOLIC20}
   \begin{aligned}
   \int_U   \beta(u(t,x)) \, \vp(x) \,dx &= \int_U \beta(u_0(x)) \, \vp(x) \ dx 
\\[5pt]
&+ \int_{0}^{t} \int_U  \beta(u(s,x)) \, b^i(s,x) \, \partial_i  \vp(x)  \ dx  ds 
\\[5pt]
&+ \int_{0}^{t} \int_U  \beta(u(s,x)) \, \dive b(s,x) \, \vp(x) \ dx ds 
\\[5pt]
&+  \int_{0}^{t}  \int_U  \beta(u(s,x)) \, \partial_i \vp(x) \, dx \circ dB_{s}^{i},
\end{aligned}
\end{equation}
where we have used the Dominated Convergence Theorem.

3. Recall Remark \ref{lemmaito1} and taking the expectation, 
it follows from \eqref{EQPARABOLIC20} that, the function 
$v(t,x)=\mathbb{E}(\beta(u(t,x)))$ satisfies 
$$
   \begin{aligned}
   \int_U   v(t,x) \, \vp(x) \,dx &= \int_U \beta(u_0(x)) \, \vp(x) \ dx 
\\[5pt]
&+ \int_{0}^{t} \int_U  v(s,x) \, b^i(s,x) \, \partial_i  \vp(x)  \ dx  ds 
\\[5pt]
&+ \int_{0}^{t} \int_U  v(s,x) \, \dive b(s,x) \, \vp(x) \ dx ds 
\\[5pt]
&+  \frac{1}{2} \int_{0}^{t}  \int_U  v(s,x) \, \Delta \vp(x) \, dx ds.
\end{aligned}
$$
Finally, for $\zeta \in C_c^{\infty}([0,T))$ we multiply the above equation by $\zeta'(t)$,
and integrating in $[0,T)$, we obtain that
$$
   \begin{aligned}
  \int_0^T  \int_U   v(t,x) \,\zeta'(t)  \vp(x) \,dxdt  &= -\int_U \beta(u_0(x)) \, \zeta(0) \vp(x) \ dx 
\\[5pt]
&- \int_{0}^{T} \int_U  v(s,x) \, b^i(t,x) \, \zeta(t) \partial_i  \vp(x)  \ dx  dt 
\\[5pt]
&- \int_{0}^{T} \int_U  v(t,x) \, \dive b(t,x) \, \zeta(t) \vp(x) \ dx dt
\\[5pt]
&-  \frac{1}{2} \int_{0}^{T}  \int_U  v(t,x) \, \zeta(t) \Delta \vp(x) \, dx dt.
\end{aligned}
$$
Since finite sums of function $\zeta_i(t) \vp_i(x)$,   
$(\zeta_i \in C_c^\infty([0,T)), \vp_i \in C_c^\infty(U))$ are dense in the space 
of test functions $\clg{D}([0,T) \times U)$, the thesis of the lemma follows by 
a standard density argument.
\end{proof}

Next, we pass to the uniqueness theorem. 
\begin{theorem}
\label{Tuni} Let $b$ be a drift vector field
satisfying conditions \eqref{con1}, \eqref{BCONDUNQ}.
If $u,v \in L^\infty(U_T \times \Omega)$ are two 
weak  $L^{\infty}-$solutions of the  IBVP \eqref{trasport}, 
with the same initial-boundary data
$u_0 \in L^\infty(U)$,  $u_b \in L^\infty(\Gamma_T)$, then $u \equiv v$ almost
sure in $U_T \times \Omega$.
\end{theorem}

\begin{proof}
1. First, by linearity it is enough to show that,
a  weak $L^{\infty}-$solution of the IBVP 
\eqref{trasport}, say $u(t,x)$,  with initial-boundary condition $u_{0}= 0$ and  $u_b= 0$ vanishes
identically. Since $u$ is a weak solution, for each $\vp \in C_c^\infty(\R^d)$, and $t \in [0,T]$, we have 
\begin{equation}
\label{IBVPEXTUN}
     \begin{aligned}
      \int_U  & u(t,x) \varphi(x) dx=
      \int_{0}^{t} \!\! \int_{U}u(s,x) \, b \cdot \nabla \varphi(x) \ dx ds
\\[5pt]     
        &+ \int_{0}^{t}\!\!  \int_{U} u(s,x) \, \dive \, b(s,x)\,  \varphi(x) \ dx ds 
           -  \int_{0}^{t} \!\! \int_{\Gamma} \gamma u(s,r) \, \varphi(r) \, b \cdot \mathbf{n} \ drds
\\[5pt]
     & -\int_{0}^{t} \!\! \int_{\Gamma}  \gamma u(s,r) \, \varphi(r)  \, \mathbf{n}_j \ d{r}\,  {\circ}dB_{s}^{j}
      + \int_{0}^{t} \!\! \int_{U} u(s,x)  \; \partial_{x_{j}} \varphi(x) \ dx \, {\circ}dB_{s}^{j},
   \end{aligned}
\end{equation}
where $\gamma u \, \mathbf{n}= u_\mathbf{o} \, \mathbf{n^o}$, since $u_b= 0$.
In particular, taking $\vp \in C^\infty_c(U)$, it follows that $u$ is a distributional 
$L^\infty$-solution of the stochastic IBVP \eqref{trasport}. 
Then, we may extended $u(t,x)$ by zero for $x \in \R^d \setminus U$,
and apply Lemma \ref{leuni} to obtain, for all $\psi \in C^\infty_c([0,T) \times \R^d)$, and
any $\beta \in C^2(\R)$, with $\beta(0)= 0$, that $v(t,x)= \mathbb{E}(\beta(u(t,x)))$ satisfies 
\begin{equation}
\label{IBVPEXTUNQ}
   \begin{aligned}
  \int_0^T  \int_{\R^d}   v(t,x) \, \partial_t  \psi(t,x) \,dxdt &= - \int_{0}^{T} \int_{\R^d}  v(t,x) \, b^i(t,x) \, \partial_i  \psi(t,x)  \ dx dt 
\\[5pt]
&- \int_{0}^{T} \int_{\R^d}  v(t,x) \, \dive b(t,x) \, \psi(t,x) \ dx dt 
\\[5pt]
&-  \frac{1}{2} \int_{0}^{T}  \int_{\R^d}  v(t,x) \, \Delta \psi(t,x) \, dx dt.
\end{aligned}
\end{equation}

2. Consider by condition 
\eqref{BCONDUNQ} a non-negative function $\alpha \in L^1_\loc(\R)$ 
such that, $|b(t,x)| \leq \alpha(t)$ almost everywhere.
Then, for each $\theta > 0$ by Lusin's Theorem (see
Evans-Gariepy \cite{EG}, Section 1.2), there exists a compact set
$\clg{I}_\theta \subset [-2T,2T]$, such that,
$\clg{H}^1([-2T,2T]-\clg{I}_\theta) < \theta$ and
$\alpha|_{\clg{I}_\theta} =:\alpha_\theta$ is a non-negative
continuous function. Thus, we define $k_\theta:= \max_{t \in
\clg{I}_\theta} \alpha_\theta(t)$.

The main issue is to consider a non-negative function $\vp(t,x)$
with compact support, which satisfies 
\begin{equation}
\label{HJED}
  \partial_t \vp(t,x) + k_\theta \; |\nabla \vp(t,x)|
  + \frac{1}{2} \Delta \vp(t,x) \leq 0.
\end{equation}
Fix $t_0 \in [0,T]$ and choose a non-negative function $\zeta \in C_c^\infty([0,\infty))$, such that
$$
   \zeta' \leq 0, \quad 0 \leq \zeta'' \leq \frac{-\zeta'}{R},
$$
where $R> 0$ is the diameter of the support of $\zeta$. Then, we define 
$$
    \vp(t,x):= \zeta(k_\theta |t-t_0| + |x|), 
$$
and observe that \eqref{HJED} is satisfied for a.e. $(t,x) \in (-\infty, t_0) \times \R^d$. 

Now,  let $\chi \in C_c^\infty\big([0,2T) \big)$ be a non-negative test function. 
Then, taking $\psi(t,x)= \chi(t) \, \vp(t,x)$ in \eqref{IBVPEXTUNQ} we have
\begin{equation}
\label{UNIQINEQ}
      \begin{aligned}
  &\int_0^T  \int_{\R^d} v(t,x) \, \chi'(t) \, \vp(t,x) \,dx dt 
\\[5pt]  
  &= - \int_{0}^{T} \int_{\R^d}  v(t,x) \chi(t) \Big( \partial_t \vp(t,x) + b(t,x) \cdot  \, \nabla  \vp(t,x)  + \frac{1}{2} \Delta \vp(t,x) \Big) \, dx  dt
\\[5pt]
& \quad - \int_{0}^{T} \int_{\R^d}  v(t,x) \, \dive b(t,x) \, \chi(t) \, \vp(t,x) \ dx dt
\\[5pt]  
  &\geq - \int_{[0,T] \cap \clg{I}_\theta^c} \int_{\R^d}  v(t,x) \chi(t) \Big( \partial_t \vp(t,x) + \alpha(t) |\nabla  \vp(t,x)|  + \frac{1}{2} \Delta \vp(t,x) \Big)  dx  dt
\\[5pt]
& \quad - \int_{0}^{T} \int_{\R^d}  v(t,x) \, \gamma(t) \, \chi(t) \, \vp(t,x) \ dx dt,
\end{aligned}
\end{equation}
where we have used \eqref{BCONDUNQ} and the above assumptions on $\vp$. 
Hence we take $\chi(t)$ be the characteristic function
of the interval $[\delta,t_0-\delta]$ for any $\delta>0$ (sufficiently small).  
Therefore, passing to the limit as $\delta \to 0$ and also $\theta \to 0$, we obtain 
 from equation \eqref{UNIQINEQ}
$$
    \int_{\R^d} v(t,x) \, \vp(t,x) \,dx \leq \int_0^T  \gamma(t) \int_{\R^d} v(t,x) \, \vp(t,x) \ dx dt. 
$$
Applying the Gronwall Inequality,
we obtain that $v(t,x)= 0$ a.e.. Thus taking $\beta(z)=z^{2}$, 
we conclude that  $u= 0$ almost sure in $U_T \times \Omega$.

3. Finally, since $u= 0$ almost sure in $U_T \times \Omega$, it follows from \eqref{IBVPEXTUN} for any test function 
$\vp \in C^\infty_c(\R^d)$, and all $t \in [0,T]$
\begin{equation}\label{euauni}
  \int_{0}^{t} \int_{\Gamma} \gamma u(s,r) \, \vp(r) \ b \cdot \mathbf{n}  \ dr ds
  + \int_{0}^{t} \int_{\Gamma} \gamma u(s,r) \, \vp(r)\,  \mathbf{n}_i
  \ dr \circ dB_{s}^{i}= 0.
\end{equation}
Therefore, taking the covariation with respect to $B^j$, we obtain
$$
   \int_{0}^{t} \int_{\Gamma} \gamma u(s,r) \, \vp(r)\,  \mathbf{n}_j  \ dr ds= 0, \quad (\forall j=1,\ldots,d),
$$
which implies that $\gamma u= 0$ almost sure in $[0,T] \times \Gamma \times \Omega$.
\end{proof}

\section{Appendix} 

At this point we fix some notation and material used through of this paper.

\medskip
Let us fix a stochastic basis with a
$d$-dimensional Brownian motion 
$$
   \big( \Omega, \mathcal{F}, \{
   \mathcal{F}_t: t \in [0,T] \}, \mathbb{P}, (B_{t}) \big).
$$   
Then, we recall to help the intuition, the following definitions 
$$
\begin{aligned}
\text{It\^o:}&  \ \int_{0}^{t} X_s dB_s=
\lim_{n    \rightarrow \infty}   \sum_{t_i\in \pi_n, t_i\leq t}  X_{t_i}    (B_{t_{i+1} \wedge t} - B_{t_i}),
\\[5pt]
\text{Stratonovich:}&  \ \int_{0}^{t} X_s  \circ dB_s=
\lim_{n    \rightarrow \infty}   \sum_{t_i\in \pi_n, t_i\leq t} \frac{ (X_{t_{i+1} \wedge t   } + X_{t_i} ) }{2} (B_{t_{i+1} \wedge t} - B_{t_i}),
\\[5pt]
\text{Covariation:}& \ [X, Y ]_t =
\lim_{n    \rightarrow \infty}   \sum_{t_i\in \pi_n, t_i\leq t} (X_{t_{i+1} \wedge t   } - X_{t_i} )  (Y_{t_{i+1} \wedge t} - Y_{t_i}),
\end{aligned}
$$
where $\pi_n$ is a sequence of finite partitions of $ [0, T ]$ with size $ |\pi_n| \rightarrow 0$ and
elements $0 = t_0 < t_1 < \ldots  $. The limits are in the sense of probability, and uniformly in time
on compact intervals. Details about these facts can be found in Kunita 
\cite{Ku2}. Also we address from that book, 
Itô's formula, the chain rule for the stochastic integral, 
for any continuous d-dimensional semimartingale $X = (X_1,X_2,\ldots,X_d)$, 
and twice continuously differentiable and real valued function f on   $\mathbb{R}^{d}$.

\medskip
Let $U \subset \R^{n}$ be an open set, and $\Gamma$ its boundary.
A map $\Psi:[0,1] \times \Gamma \to \ol{U}$ is said an 
admissible deformation, when satisfies the following conditions:
\begin{enumerate}
\item[$(1)$] For all $r \in \Gamma$, $\Psi(0,r)=r$.

\item [$(2)$] The derivative of the map $[0,1] \ni \tau \mapsto \Psi(\tau, r)$
at $\tau= 0$ is not orthogonal
to $\mathbf{n}(r)$, for each $r \in \Gamma$.
\end{enumerate}
Moreover, for each $\tau\in[0,1]$, we denote:
$\Psi_\tau$ the mapping from $\Gamma$ to $\ol{U}$, given by $\Psi_\tau(x):=\Psi(\tau,x)$;
$\Gamma^\tau= \Psi_\tau(\Gamma)$;
$\mathbf{n}^\tau$ the unit outward normal field in $\Gamma^\tau$. In particular,
$\mathbf{n}^0(r)= \mathbf{n}(r)$ is the unit outward normal field in $\Gamma$.

Now, we define a level set function $h$ associated with the 
deformation $\Psi_\tau$.
For $\delta> 0$ sufficiently small we define
$$
   h(x):= \left \{ 
         \begin{aligned}
           \min\{\tau,\delta\}, &\quad \text{if $x \in U$},
            \\[5pt]
           - \min\{\tau,\delta\},& \quad \text{if $x \in \Bbb{R}^n \setminus U$}.           
         \end{aligned}\right.
$$
The function $h(x)$ is Lipschitz continuous in $\Bbb{R}^n$, and $C^2$ on the
closure of $\left\lbrace x \in \Bbb{R}^n: |h(x)|< \delta \right\rbrace$, see Gilbarg, Trudinger \cite{DGNST}, p. 355.

\bigskip
Given a function  $f \in L^1(U)$, we recall the global approximation by smooth functions, that is, $f_\varepsilon \in L^1(U) \cap C^\infty(\ol{U})$, 
such that, $f_\ve \to f$ in $L^1$, see Evans, Gariepy \cite{EG} Chapter 4.2, Theorem 1 and Theorem 3. In fact, this result follows from a convenient modification of the 
standard mollification of $f$ by a standard (symmetric) mollifier $\rho$, that is a positive radial and regular function with compact support in $\R^d$, 
such that $\int \rho(x) dx= 1$. For each $\ve> 0$, we define $\rho_\ve(x):= \ve^{-d} \rho(\frac{x}{\ve})$.
For convenience, that is to fix the notation, let us give the main idea. 
For any $\ve> 0$ fixed, $0 \leq \delta \leq \ve$, and $y \in \ol{U}$, we define 
$$
    y^\ve:= y + \lambda \, \ve \, \nabla h(y),
$$ 
for $\lambda > 0$ sufficiently large. Then, we take a standard  mollifier $\rho_{\varepsilon}$, and for any  $u \in  L_{\loc}^{1}(U_T) $, we define
the following (space) global approximation
$$
    u_{\varepsilon}(t,y) \equiv (u \ast_{\mathbf{n}} \rho_\ve) (t,y):= \int_{U} u(t,z) \rho_{\varepsilon}(y^\ve-z) \ dz.
$$
Therefore, $u_\ve \in L^1_\loc([0,T]; C^\infty(\ol{U}))$ and converges to $u$ in $L^1_\loc$. 

\section*{Acknowledgements}

Conflict of Interest: The author 
Wladimir Neves has received research grants
from CNPq through the grant 308064/2019-4,
and also by FAPESP through the grant 2013/15795-9. 
C. Olivera  is partially supported by FAPESP 
by the grants 2017/17670-0 and   2015/07278-0.
		

\end{document}